\documentclass[11pt]{article}
\usepackage{amsmath, amssymb}
\usepackage{amsfonts}
\usepackage[utf8]{inputenc}
\usepackage{amsthm}
\usepackage{graphicx}
\usepackage{csquotes}
\usepackage{xspace}
\usepackage{fullpage}

\graphicspath{{picss/}}

\newtheorem{theorem}{Theorem}[section]
\newtheorem{definition}[theorem]{Definition}
\newtheorem{lemma}[theorem]{Lemma}
\newtheorem{corollary}[theorem]{Corollary}
\newtheorem{proposition}[theorem]{Proposition}
\newtheorem{remark}[theorem]{Remark}

\renewcommand\appendix{\par
  \setcounter{section}{0}
  \setcounter{subsection}{0}
  \setcounter{figure}{0}
  \setcounter{table}{0}
  \renewcommand\thesection{\Alph{section}}
  \renewcommand\thefigure{\Alph{section}\arabic{figure}}
  \renewcommand\thetable{\Alph{section}\arabic{table}}
}

\newcommand{\NP}{{\ensuremath{\mathbf{NP}}}\xspace}

\newcommand{\WL}{\mathrm{WL}}
\newcommand{\EP}{\mathrm{EP}}
\newcommand{\C}{\mathrm{C}}
\newcommand{\IM}{\mathrm{IM}}
\newcommand{\Sch}{\mathrm{Sch}}

\newcommand{\Sym}{\mathrm{Sym}}

\newcommand{\nats}{\mathbb{N}}
\newcommand{\pr}{\mathrm{pr}}
\newcommand{\im}{\mathrm{Im}}

\newcommand{\mult}{\mathrm{Mult}}
\newcommand{\IMt}{\mathrm{IMt}}
\newcommand{\Mat}{\mathrm{Mat}}

\begin{document}

\title{\textbf{Generalizations of $k$-dimensional Weisfeiler-Leman stabilization
}}
\author{Anuj Dawar \and Danny Vagnozzi}
\maketitle

\begin{abstract}
The family of Weisfeiler-Leman equivalences on graphs is a widely
studied approximation of graph isomorphism with many different
characterizations.  We study these, and other approximations of isomorphism defined in terms of refinement operators and Schurian Polynomial Approximation Schemes (SPAS).  The general framework of SPAS allows us to study a number of parameters of the refinement operators based on Weisfeiler-Leman refinement, logic with counting, lifts of Weisfeiler-Leman as defined by Evdokimov and Ponomarenko, and the invertible map test introduced by Dawar and Holm, and variations of these, and establish relationships between them.
\end{abstract}

\section{Introduction}
For convenience, we shall treat graphs as \emph{arc-coloured complete digraphs};  that is to say,  as \emph{labelled partitions} of the set of ordered pairs of vertices (hereafter, we refer to the latter as \emph{arcs}).  For example, an undirected simple graph can be seen as a partition of its arcs into edges, non-edges and loops.  As such, the \emph{graph isomorphism problem} is that of deciding whether there is a colour-preserving bijection between the sets of vertices of two given graphs.  Computationally, this problem is polynomial-time equivalent to finding the orbits of the induced action of the automorphism group of a given graph $G$ on a fixed power of its vertex set $V$ \cite{mathon}.  For short, we refer to the partition of $V^k$ (for any fixed $k$) obtained by this action as the \emph{orbit partition} of $V^k$.
 The graph isomorphism problem (and likewise the problem of determining the orbit partition) is neither known to be solvable in polynomial time nor known to be $\NP$-complete. The best known upper bound to their computational time is quasi-polynomial. This follows from the well-known result by Babai~\cite{Babai16}.

The \emph{classical Weisfeiler-Leman} (WL) algorithm is a well known method for approximating the orbits of the induced action of the automorphism group of a given graph on the set of pairs of arcs. It can be seen as a generalization of the so called \emph{na\"{i}ve colour refinement}.  Given a graph $G$, the WL algorithm produces a \emph{coherent configuration}, which is a partition of the set of arcs of $G$ satisfying certain stability conditions (see Section \ref{cohs} for the definition). A natural generalization of this algorithm was given by Babai: for each $k\in \mathbb{N}$, the $k$-\emph{dimensional Weisfeiler-Leman} ($\WL_k$) algorithm outputs a labelled partition of $k$-tuples of vertices satisfying a similar stability condition and respecting local isomorphism. The running time of the $\WL_k$ algorithm  on a graph with $n$ vertices is bounded by $n^{O(k)}$.   The case $k=1$ coincides with the na\"ive colour refinement, and $k=2$ with the classical Weisfeiler-Leman algorithm.

It follows from a result by Cai, F\"urer and Immerman \cite{cfi} that there is no fixed $k\in \mathbb{N}$ such that for all graphs the $k$-dimensional Weisfeiler-Leman algorithm outputs the partition of $k$-tuples of vertices into orbits of the induced action of the automorphism group of the input graph.  Indeed, the authors show how to construct a graph with $O(k)$ vertices for which $\WL_k$ fails to produce the partition into such orbits.  Thus, their result implies that a partition induced by this group action can be obtained for all  graphs on $n$ vertices only if one chooses $k$ to be $\Omega(n)$.  One can informally claim that the strength of the $k$-dimensional Weisfeiler-Leman algorithm increases with $k$. More precisely, for unlabelled partitions $\mathcal{P}$ and $\mathcal{Q}$ of some set $A$ we write $\mathcal{P}\preceq _A \mathcal{Q}$ and say  $\mathcal{Q}$ is a refinement of $\mathcal{P}$ if, whenever $a,b \in A$ are in the same equivalence class of $\mathcal{Q}$, they are also in the same equivalence class of $\mathcal{P}$. By viewing a labelled partition of $A$ as a function $\gamma: A \rightarrow L$ to a set of labels $L$ (which we sometimes refer to as the \emph{colour} set), the unlabelled partition induced by $\gamma$ is $\{\gamma^{-1}(l) \mid l \in L\}$.  We extend the partial order $\preceq_A$ to labelled paritions by writing $\gamma \preceq_A \rho$ to mean that the unlabelled parition induced by $\rho$ refines that induced by $\gamma$.  Note that this does not require that the co-domains of $\gamma$ and $\rho$ are the same.  We omit the subscript $A$ where the set is clear from the context. For a graph $\Gamma$, define $\overline{\WL}_1(\Gamma)=\Gamma$ and for $k\geq 2$ set $\overline{\WL}_k(\Gamma)$ to be the labelled partition of the set of arcs induced by the output of the $k$-dimensional Weisfeiler-Leman algorithm on input $\Gamma$.  We can now state the following:
$$
\overline{\WL}_1(\Gamma)\preceq \overline{\WL}_2(\Gamma)\preceq \hdots \preceq\overline{\WL}_n(\Gamma)=\overline{\WL}_{n+1}(\Gamma)=\hdots = \overline{\WL}_\infty(\Gamma)
$$
where $n$ is the number of vertices of $\Gamma$ and $\overline{\WL}_\infty(\Gamma)$ is the partition into the orbits of the induced action of the automorphism group on arcs. Also,
$$
\overline{\WL}_l(\overline{\WL}_k(\Gamma))=\overline{\WL}_k(\Gamma)
$$
for all $l,k\in \mathbb{N}$ with $l\leq k$.  This shows that the family of maps from the set of arc-coloured complete digraphs to itself $\{\overline{\WL}_1,\overline{\WL}_2,\hdots\}$ forms a \emph{Schurian polynomial approximation scheme} in the following sense, as defined in~\cite{pono}.
\begin{definition}[Schurian polynomial approximation scheme]\label{def:SPAS}
We say that a family of mappings $\{X_1,X_2,\hdots\}$ forms a \emph{Schurian Polynomial Approximation Scheme} (SPAS) if for any graph $\Gamma$ with vertex set $V$
\begin{enumerate}
\item $X_k(\Gamma)$ is a graph with vertex set $V$ for all $k\in \mathbb{N}$.
\item $X_1(\Gamma)\preceq X_2(\Gamma)\preceq \hdots \preceq X_{n}(\Gamma)=X_{n+1}(\Gamma)=\hdots=X_{\infty}(\Gamma)=Sch(\Gamma)$, where $n=|V|$ and $Sch(\Gamma)$ is the partition of arcs into orbits of the induced action of the automorphism group of $\Gamma$.
\item $X_l(X_m(\Gamma))=X_m(\Gamma)$ for all $l,m\in \mathbb{N}$ with $l\leq m$.
\item $X_k(\Gamma)$ is computable in time $n^{O(k)}$.
\end{enumerate}
\end{definition}
More informally, one can think of SPAS as a collection of polynomially computable functions indexed by $\mathbb{N}$, each of which takes as input a partition and refines it.  Moreover, the higher the index, the closer the output parition is to the orbit partition of the set of arcs. 
\begin{definition}[Dominance and equivalence of SPAS]\label{SPAS2}
For any two SPAS $\mathcal{S}_X=\{X_1,X_2,\hdots\}$ and $\mathcal{S}_Y=\{Y_1,Y_2,\hdots\}$:
\begin{enumerate}
\item we say $\mathcal{S}_X$ \emph{dominates} $\mathcal{S}_Y$ and write $\mathcal{S}_Y\preceq \mathcal{S}_X$ if for each $k\in\nats$ there is some $k'\in\nats$ such that $Y_{k}(\Gamma)\preceq X_{k'}(\Gamma)$ for all graphs $\Gamma$;
\item we say $\mathcal{S}_X$ is \emph{equivalent} to $\mathcal{S}_Y$ and write $\mathcal{S}_X\simeq \mathcal{S}_Y$ if $\mathcal{S}_X \preceq \mathcal{S}_Y$ and $\mathcal{S}_Y \preceq \mathcal{S}_X$; and 
\item we say $\mathcal{S}_X$ \emph{strictly dominates} $\mathcal{S}_Y$ if $\mathcal{S}_Y \preceq \mathcal{S}_X$ but $\mathcal{S}_X\not\preceq \mathcal{S}_Y$.
\end{enumerate}
\end{definition}

In this paper we deal with the following SPAS: $\mathcal{S}_{\mathbf{WL}}$, $\mathcal{S}_{\mathbf{C}}$, $\mathcal{S}_{\mathbf{C},r}$, $\mathcal{S}_{\mathbf{WL},r}$ and $\mathcal{S}_{\mathbf{IM}(\mathbb{F})}$. Such schemes arise from considering fixed points of \emph{refinement operators}. The formal definitions of \emph{refinement operators} is given in Section \ref{rop}.  These are operators that take a partition $\gamma$ of $V^k$ to a refinement of itself. The SPAS $\mathcal{S}_{\mathbf{WL}}$ and $\mathcal{S}_{\mathbf{C}}$ arise from well known concepts: Babai's generalization of the classical Weisfeiler-Leman algorithm for the former and first order logic with counting quantifiers for the latter.  In each case, the label of a tuple $\vec{v} \in V^k$ in the refined partition is determined by its label in $\gamma$ and the partition of $V$ that is induced by considering $\gamma(\vec{u})$ for the tuples $\vec{u}$ obtained by substituting elements of $V$ in $\vec{v}$.  
For $r\in \mathbb{N}$, $\mathcal{S}_{\mathbf{WL},r}$ and $\mathcal{S}_{\mathbf{C},r}$ are further generalizations of $\mathcal{S}_{\mathbf{WL}}$ and $\mathcal{S}_{\mathbf{C}}$ respectively.   In these generalzations, the label associated to each $k$-tuple $\vec{v}$ is determined by $\gamma(\vec{v})$ and the partition of $V^r$ obtained by considering  $\gamma(\vec{u})$ for tuples $\vec{u}$ obtained by substituting $r$-tuples in $\vec{v}$.  Formal definitions of these are given later.  Here we note that our first result is that the parameter $r$ does not strengthen the SPAS $\mathcal{S}_{\WL}$ and $\mathcal{S}_{\C}$.
\begin{theorem}\label{thm:main1}
For any $r\in \mathbb{N}$, $\mathcal{S}_\mathbf{WL}\simeq \mathcal{S}_\mathbf{C} \simeq \mathcal{S}_{\mathbf{WL},r}\simeq \mathcal{S}_{\mathbf{C},r}$.
\end{theorem}
The reasons for considering the additional parameter $r$ is that it appears to be of interest in another scheme we consider.
The scheme $\mathcal{S}_{\mathbf{IM}(\mathbb{F})}$ arises from the \emph{invertible map game} introduced in~\cite{DawarHolm}.  It has been shown to have a close relationship to logics with linear algebraic operators~\cite{DawarGraedelPakusa} over a field $\mathbb{F}$.  The associated refinement operator $\IM_k^{\mathbb{F}}$ maps each $k$-tuple of vertices $\vec{v}$ and partition $\gamma$ to a tuple of matrices from, and the colour associated to $\vec{v}$ by the refinement of $\gamma$ is determined by the equivalence class of this tuple of matrices under simultaneous similarity.

For this SPAS we prove the following results.
\begin{theorem}\label{thm:main2}
For any field $\mathbb{F}$ of characteristic $0$, $\mathcal{S}_{\mathbf{IM}(\mathbb{F})}\simeq\mathcal{S}_{\mathbf{WL}}$.
\end{theorem}
\begin{theorem}\label{thm:main3}
For any field $\mathbb{F}$ of positive characteristic, $\mathcal{S}_{\mathbf{IM}(\mathbb{F})}$ strictly dominates $\mathcal{S}_\mathbf{WL}$.
\end{theorem}

The paper is structured as follows: after a brief overview of the required notions on coherent configurations and algebras, we formally define and discuss the concepts of refinement operators and procedures. We then prove Theorem $\ref{thm:main1}$ and use a similar method to show the equivalence between $\mathcal{S}_\mathbf{WL}$ and $\mathcal{S}_\EP$, a SPAS introduced by Evdokimov and Ponomarenko in \cite{pono}.
The final two sections contain the proofs of Theorems $\ref{thm:main2}$ and $\ref{thm:main3}$ and a short discussion on a variant of the SPAS $\mathcal{S}_{\mathbf{IM}(\mathbb{F})}$, namely $\mathcal{S}_{\mathbf{IMt}(\mathbb{F})}$.  This variant is motivated by looking at the difference between the definitions of $\mathcal{S}_\mathbf{WL}$ and $\mathcal{S}_{\mathbf{C}}$, and applying a similar variation to the definition of $\mathcal{S}_{\mathbf{IM}(\mathbb{F})}$.  The discussion leads to a proof of the following.
\begin{theorem}\label{thm:main4}
For any field $\mathbb{F}$, $\mathcal{S}_{\mathbf{IM}(\mathbb{F})}\simeq\mathcal{S}_{\mathbf{IMt}(\mathbb{F})}$.
\end{theorem}

Throughout the text, all sets are finite. Given two sets $V$ and $I$, a tuple in $V^I$ is denoted by $\vec{v}$, and its $i^{th}$ entry by $v_i$, for each $i\in I$. We use the notation $(v_i)_{i\in I}$ to denote the element of $V^I$ with $i^{th}$ element equal to $v_i$. We set $[k]=\{1,2,\hdots,k\}\subset\mathbb{N}$ and $[k]^{(r)}=\{\vec{x}\in [k]^r \mid x_i\neq x_j\;\forall i,j\in [r],\;i\neq j\}$.  Recall that a labelled partition of a set $A$ is a function $\gamma:A\rightarrow \mathrm{Im}(\gamma)$.  The class of all labelled partitions of $A$ is denoted by $\mathcal{P}(A)$.  Recall also that  for $\gamma,\rho\in \mathcal{P}(A)$, $\gamma\preceq \rho$ denotes that the unlaballed partition induced by $\rho$ is a refinement of that induced by $\gamma$. If $\gamma\preceq \rho$ and $\gamma\succeq \rho$ are both satisfied, we write $\gamma \approx \rho$. Note that $\gamma\approx \rho$ does not imply $\gamma=\rho$, as they may have different codomains. The equivalence class of $a\in A$ with respect to the partition $\gamma$ is denoted by $[a]_\gamma$. Fix some set $V$ and $k,r\in \mathbb{N}$ with $r\leq k$. For any $\vec{v}\in V^k$, $\vec{u}\in V^r$, and $\vec{i}\in [k]^{(r)}$ we define $\vec{v}\langle \vec{i},\vec{u}\rangle\in V^k$ to be the tuple with entries
$$
(\vec{v}\langle \vec{i},\vec{u}\rangle)_j=\begin{cases}
u_{i_s}\;\,\text{if $j=i_s$ for some $s\in [r]$}\\
v_{j}\;\,\text{otherwise.}\\
\end{cases}
$$
Given two tuples $\vec{v}\in V^r$ and $\vec{w}\in V^s$ their \emph{concatenation} is denoted by $\vec{v}\cdot \vec{w}\in V^{r+s}$. More precisely, $\vec{v}\cdot\vec{w}$ is the tuple with entries
$$
(\vec{v}\cdot\vec{w})_i=\begin{cases}
v_i \;\,& \textrm{if $i\in [r]$}\\
w_j\;\,& \textrm{if $i=j+r$}.
\end{cases}
$$
For a relation $R\subseteq V^2$ we define the adjacency matrix of $R$ to be the $V\times V$ matrix whose $(u,v)$ entry is $1$ if $(u,v)\in R$ and $0$ otherwise. The set of multisets of elements of $V$ is denoted by $\mult(V)$, and the multiset of entries of a tuple $\vec{v}\in V^I$ is denoted by $\{\{v_i \mid i\in I\}\}$. For all $\gamma\in\mathcal{P}(V^k)$ and natural numbers $r\leq k$, set $\Phi^{\gamma,r}=\im(\gamma)^{[k]^{(r)}}$.

\section{Coherent configurations and coherent algebras}\label{cohs}
This section introduces notions on coherent configurations and algebras necessary throughout the paper. For a more in-depth account see \cite{pono2} or \cite{cameron2}.  Our formulation is, in general, different from the more traditional treatment, as we deal with labelled partitions and extend the notion of coherent algebras to arbitrary fields. Also note that rainbows and coherent configurations have been originally defined for \emph{unlabelled partitions}.  Thus, Definitions \ref{cohrain} and \ref{cohconf} define, strictly speaking, a \emph{labelled} rainbow and a \emph{labelled} coherent configuration respectively.

\begin{definition}[Rainbow]\label{cohrain}
A labelled partition $\rho$ of $V^2$ is said to be a \emph{rainbow} on $V$ if
\begin{enumerate}
\item There is a set $\mathcal{I}\subseteq \im(\rho)$ such that
\begin{equation}\label{rain1}
\bigcup_{\sigma\in \mathcal{I}} \{\vec{x}\in V^2 \mid \rho(\vec{x})=\sigma\}=\{(v,v)\in V^2 \mid v\in V\}.
\end{equation}
\item For all $(u,v),(u',v')\in V^2$, $\rho (u,v)=\rho(u',v')\iff\rho(v,u)=\rho(v',u').$
\end{enumerate}
\end{definition}

We set $\text{Cel}(\rho)=\{ U\subset V \mid \exists \sigma\in \im(\rho),\rho(u,u)=\sigma\,\,\forall u\in U\}$ and call its elements the \emph{cells} of $\rho$.

It was stated in the introduction that in this paper, graphs are viewed as partitions of the set of their arcs, hence as arc-coloured complete di-graphs.  For example, an uncoloured loop-free undirected graph can be seen as a complete di-graph with its arcs partitioned into three colour classes: edges, non-edges and loops.  Hence we can always see a graph as a rainbow in the above sense.  This view is natural, since our interest is in the partition into orbits of the induced action of the automorphism group on arcs, and this partition is, necessarily, a rainbow.  Furthermore, given any group action on $V$, the partition into the orbits of the induced action on $V^2$ forms a \emph{coherent configuration} \cite{cameron2}\footnote{Though note that although all group actions give rise to coherent configurations, not all coherent configurations arise from group actions (see \cite{cameron2}).}.

\begin{definition}[Coherent configuration]\label{cohconf}
  A rainbow $\rho$ on $V$ is said to be a \emph{coherent configuration} on $V$ if for each $\sigma,\tau,\kappa \in \im(\rho)$, there is a constant $p_{\sigma\tau}^\kappa$ such that for any $(u,v)\in V^2$ with $\rho(u,v)=\kappa$:
$$
|\{x\in V \mid \rho(u,x)=\sigma,\rho(x,v) = \tau\}| = p_{\sigma\tau}^\kappa.
$$
\end{definition}

Observe that if $X$ is a union of cells of a coherent configuration, the restriction $\rho |_{X^2}$ is a coherent configuration on $X$.

The constants $p^\kappa_{\sigma\tau}$ are called the \emph{intersection numbers} of $\rho$ and may be interpreted algebraically as follows. For every $\sigma\in \im(\rho)$, let $A_\sigma$ be the adjacency matrix of the relation $\rho^{-1}(\sigma)$. Then for all $\sigma,\tau\in \im(\rho)$
$$
A_\sigma A_\tau= \sum_{\kappa\in\im(\rho)}p^\kappa_{\sigma\tau}A_{\kappa}.
$$
Thus, taking $p^\kappa_{\sigma\tau}$ as rational numbers in a field  $\mathbb{F}$ of characteristic zero, we see that the $\mathbb{F}$-span of the set $\mathcal{A}_\rho=\{A_\sigma \mid \sigma\in\im(\rho)\}$ is an $\mathbb{F}$-algebra.  The same is true if we take $\mathbb{F}$ to be a field of characteristic $q$ and consider the constants $p^\kappa_{\sigma\tau}$ modulo $q$.  We refer to this algebra as the $\mathbb{F}$-\emph{adjacency algebra of} $\rho$ and denote it by $\mathbb{F}\mathcal{A}_\rho$. Such an algebra is a \emph{coherent algbera} in the following sense.

\begin{definition}[Coherent algebra]\label{coh2}
A subalgebra of $\text{Mat}_V(\mathbb{F})$ is said to be a \emph{coherent algebra} on $V$ if it is a unital algebra with respect to matrix multiplication and Schur-Hadamard (component-wise) multiplication, and it is closed under transposition.
\end{definition}

We indicate the Schur-Hadamard multiplication by $\star$ \footnote{In the literature, the Schur-Hadamard multiplication is often denoted by $\circ$.  However, we reserve the latter for function composition.}. At this point, it needs to be pointed out that in most literature, when $\mathbb{F}=\mathbb{C}$, closure under transposition is usually replaced by \emph{closure under Hermitian conjugation} for the definition of a coherent algebra. However, we show in Proposition \ref{coh1}, that an algebra satisfies Definition \ref{coh2} if, and only if, it has a basis of $0$-$1$-matrices satisfying the \emph{coherence conditions} (Definition \ref{coh3}). In Section 2.3 of \cite{pono2} it is shown that an algebra over $\mathbb{C}$ satisfies Definition \ref{coh2} with closure under transposition replaced by closure under Hermitian conjugation if, and only if, it has a basis of $0$-$1$-matrices satisfying the coherence conditions. Hence, over $\mathbb{C}$, Definition \ref{coh2} is equivalent to the original one by D. Higman in ~\cite{higman}, but has the advantage that it can be extended to any field.

It is clear from the definition, that for any coherent configuration $\rho$, the $\mathbb{F}$-adjacency algebra $\mathbb{F}\mathcal{A}_\rho$ is a coherent algebra for any field $\mathbb{F}$. Indeed, the set $\mathcal{A}_\rho$ is the unique basis of $0$-$1$-matrices for $\mathbb{F}\mathcal{A}_\rho$ satisfying the \emph{coherence} conditions.
\begin{definition}\label{coh3}
A set of $0$-$1$-matrices $\mathcal{M}$ is said to satisfy the \emph{coherence conditions} if 
\begin{enumerate}
\item $\sum_{A\in \mathcal{M}}A=\mathbb{J}$, where $\mathbb{J}$ is the all $1$s matrix.
\item For some $\mathcal{I}\subseteq\mathcal{M}$, $\sum_{A\in \mathcal{I}}A=\mathbb{I}$, where $\mathbb{I}$ is the identity matrix.
\item $A^t\in \mathcal{M}$ for all $A\in \mathcal{M}$.
\end{enumerate}
\end{definition}
We now show that any coherent algebra over any field has a unique basis of $0$-$1$-matrices satisfying the coherence conditions. We refer to this basis as the \emph{standard basis} of a coherent algebra. The argument that follows is analogous to that used to prove Theorem 2.3.7 in \cite{pono2}.

Let $W$ be a coherent algebra over $V$. As explained in Section 2.3 in \cite{pono2}, one may write
\begin{equation}\label{coh4}
\mathbb{J}=\sum_{i\in[r]} E_i
\end{equation}
where $\{E_i \mid i\in[r]\}$ is the full set of primitive idempotents of $W$ with respect to the Schur-Hadamard product. In order to be an idempotent $E_i$ must be a $0$-$1$ matrix, and hence the adjacency matrix of some relation $R_i\subseteq V^2$. Since for $i\neq j$, $E_i$ and $E_j$ are orthogonal, for all $u,v\in V$, $(E_i)_{uv}=1\implies(E_j)_{uv}=0$. Thus, from formula (\ref{coh4}), it follows that $\{R_i \mid i\in[r]\}$ forms a partition of $V^2$.
\begin{lemma}\label{condcoh}
$\{E_i \mid i\in[r]\}$ as above, satisfies the coherence conditions.
\end{lemma}
\begin{proof}
Condition $(1)$ of Definition \ref{coh3} is satisfied because of formula (\ref{coh4}). Because $\mathbb{I}\in W$ is an idempotent, it can be written as a sum of primitive idempotents. Thus, $\{E_i \mid i\in[r]\}$ satisfies $(2)$ in Definition \ref{coh3}. Finally, $E_i^t$ is also a primitive idempotent, since $W$ is closed under transposition.
\end{proof}
\begin{proposition}\label{coh1}
For any field $\mathbb{F}$, a coherent algebra on $V$ over $\mathbb{F}$ has a unique basis of $0$-$1$-matrices satisfying the coherence conditions.
\end{proposition}
\begin{proof}
The set $B=\{E_i \mid i\in[r]\}$ satisfies the coherence conditions by Lemma \ref{condcoh}.

Suppose $\mathbb{F}$ is algebraically closed. Then $B$ is a basis for $W$, since $W$ is commutative with respect to the Schur-Hadamard product, and a basis of a semisimple commutative algebra over an algebraically closed field is given by the set of its primitive idempotents.

Suppose $\mathbb{F}$ is not algebraically closed. Since $B$ is a linearly independent set, there is some $B'\subseteq W$ such that $B\cup B'$ is a basis for $W$. Let $\mathbb{G}$ be the algebraic closure of $\mathbb{F}$ and consider the linear space $\mathbb{G}(B\cup B')\subseteq \Mat_V(\mathbb{G})$. By construction, $\mathbb{G}(B\cup B')$ is closed under transposition, matrix multiplication and Schur-Hadamard multiplication and is thus a coherent algebra over $\mathbb{G}$. From the above, $\mathbb{G}(B\cup B')$ must then have a basis $B''$ of $0$-$1$-matrices satisfying the coherence conditions. Since all entries of the elements of $B''$ are $0$ and $1$, $B''\subset\mathbb{F}(B\cup B')$ is a basis for $W$ as well. As $B''$ is a set of primitive orthogonal idempotents of $W$, it holds that $B''\subseteq B$. But $B$ is a linearly independent set, and $B''$ is a basis for $W$. Whence $B''=B=\{E_i \mid i\in[r]\}$.

The uniqueness of $B$ follows from formula (\ref{coh4}) which, indeed, implies that any basis of $0$-$1$ matrices satisfying the coherence conditions must be the set of primitive idempotents of $W$.
\end{proof}

We can denote a coherent algebra on $V$ over $\mathbb{F}$ as $\mathbb{F}\mathcal{A}$ where $\mathcal{A}$ is some set of $0$-$1$-matrices satisfying the coherence conditions.
It is easily seen that for any $(u,v)\in V^2$ and binary relations $S$ and $T$ on $V$ with adjacency matrices $A_S,A_T\in\Mat_V(\mathbb{F})$ respectively,
\begin{equation}\label{count}
|\{ x\in V \mid (u,x)\in S,(x,v)\in T\}|=(A_SA_T)_{uv}
\end{equation}
if $\mathrm{char}(\mathbb{F})=0$. Set $\rho_W:V^2\rightarrow [r]$ to be $\rho_W(u,v)=i$ if $(u,v)\in R_i$. Since there are constants $p_{ij}^k\in\mathbb{F}$ such that
$$
E_iE_j=\sum_{k\in[r]}p_{ij}^kE_k,
$$
it then follows that
\begin{equation}\label{mod}
p_{ij}^k=|\{ x\in V \mid \rho_W (u,x)=i, \rho_W(x,v)=j\}|
\end{equation}
is the same for all $(u,v)$ such that $\rho_W(u,v)=k$, and hence, $\rho_W$ is a coherent configuration. Otherwise, if $\mathbb{F}$ has characteristic $q>0$, formula (\ref{count}) holds modulo $q$.  In particular if $q>|V|$,  $\rho_W$ is a coherent configuration.
\begin{remark}
In the literature, coherent algebras over a field $\mathbb{F}$ of positive characteristic have usually been defined to be the $\mathbb{F}$-span of the adjacency matrices of the relations $\{\rho^{-1}(\sigma)|\sigma\in\im(\rho)\}$ for some coherent configuration $\rho$. The latter discussion thus shows that, over fields of positive characteristic, Definition \ref{cohconf} defines a potentially larger class of algebras.
\end{remark}

The most intuitive morphisms between coherent configurations arise from the algebraic setting. Let $\mathcal{A},\mathcal{A}'\subseteq \Mat_V(\mathbb{F})$ satisfy the coherence conditions.
\begin{definition}[Isomorphism of coherent algebras]
An $\mathbb{F}$-linear bijection $\psi:\mathbb{F}\mathcal{A}\rightarrow \mathbb{F}\mathcal{A}'$ is said to be an \emph{isomorphism} of coherent algebras if:
\begin{enumerate}
\item $\psi(\mathbb{I})=\mathbb{I}$.
\item $\psi(\mathbb{J})=\mathbb{J}$.
\item $\psi(AB)=\psi(A)\psi(B)$ and $\psi(A\star B)=\psi(A)\star\psi(B)$ for all $A,B\in \mathbb{F}\mathcal{A}_\rho$.
\end{enumerate}
\end{definition}
That is, $\psi$ preserves the structure of $\mathbb{F}\mathcal{A}$ both as a matrix algebra and as an algebra with respect to $\star$. As a consequence, the image under $\psi$ of an element of the standard basis of $\mathbb{F}\mathcal{A}$ must be an element of the standard basis of $\mathbb{F}\mathcal{A}'$, since the standard basis of a coherent algebra is the set of its primitive idempotents with respect to $\star$. Conversely, the $\mathbb{F}$-linear extension of any bijection between the standard bases of $\mathbb{F}\mathcal{A}$ and $\mathbb{F}\mathcal{A}'$ is a coherent algebra isomorphism provided it is also a matrix algebra isomorphism.

Let $\rho$ and $\rho'$ be coherent configurations and denote their intersection numbers by $p_{\sigma\tau}^\kappa$ and $q_{\sigma'\tau'}^{\kappa'}$ respectively.
\begin{definition}[Algebraic isomorphism]
A bijection $\phi: \im(\rho)\rightarrow \im(\rho')$ is said to be an \emph{algebraic isomorphism} if for all $\sigma,\tau,\kappa\in \im(\rho)$
$$
p^{\kappa}_{\sigma\tau}=q^{\phi(\kappa)}_{\phi(\sigma)\phi(\tau)}.
$$
\end{definition}
Thus, an algebraic isomorphism between coherent configurations induces a bijection between the standard bases of their respective adjacency algebras. Such bijection linearly extends to a coherent algebra isomorphism.

Crucial to this paper is the fact that when coherent algebras are semisimple (with respect to matrix product), isomorphisms between them assume a very simple form. Indeed, it is an easy consequence of the Skolem-N\"other Theorem that if $\psi:W_1\rightarrow W_2$ is an algebra isomorphism, where $W_1$ and $W_2$ are semisimple subalgebras of $\text{Mat}_n(\mathbb{F})$, then there is some $S\in \mathrm{GL}_n(\mathbb{F})$ such that $\psi(A)=SAS^{-1}$ for all $A\in W_1$. The following is a direct consequence of Theorem 4.1.3 in \cite{zieschang}\footnote{The author actually proves this statement for coherent algebras whose diagonal matrices are multiples of the identity matrix. However, the same argument applies to the more general case, and, in particular, to our more general notion of coherent algebras in the sense of Definition \ref{cohconf}.}.
\begin{theorem}
The Jacobson radical of a coherent algebra $\mathbb{F}\mathcal{A}$ is a subspace of the span of the elements of the standard basis whose number of non-zero entries is divisible by $\mathrm{char}(\mathbb{F})$.
\end{theorem}

For a coherent configuration $\rho$ on $V$, choose $u,v\in V$ such that $\rho(u,v)=\sigma$ and $\rho(u,u)=\tau$. It follows from formula (2.1.5) in \cite{pono2} that
$$
|\rho^{-1}(\sigma)|=|\rho^{-1}(\tau)||\{x\in V \mid \rho(u,x)=\sigma\}|.
$$
Since both factors on the right-hand side are no larger than $|V|$, it is clear that the prime factors of the size of any equivalence class of $\rho$ is no larger than $|V|$. We then deduce the following.

\begin{corollary}\label{newbound}
A coherent algebra on $V$ over $\mathbb{F}$ is semisimple with respect to matrix product if $\mathrm{char}(\mathbb{F})=0$ or $\mathrm{char}(\mathbb{F})>|V|$.
\end{corollary}

\section{Graph-like partitions}\label{gps}

In this section we describe some restrictions to be imposed on the partitions dealt with in the paper. Such restrictions are natural in the sense that they are necessary conditions to be satisfied by a partition of $k$-tuples into the orbits of an induced action of a group on $1$-tuples.

Fix some $k\in \mathbb{N}$, a set $V$ and let $\gamma\in\mathcal{P}(V^k)$. Define an action of $\Sym(k)$ on $V^k$ by letting, for each $\tau\in \Sym(k)$,  $\vec{v}^\tau$  be the element of $V^k$ with $i^{th}$ entry $v_{\tau^{-1}(i)}$.
\begin{definition}[Invariance]
$\gamma$ is \emph{invariant} if $\gamma(\vec{u})=\gamma(\vec{v})\implies \gamma(\vec{u}^\tau)=\gamma(\vec{v}^\tau)$ for all $\vec{u}, \vec{v}\in V^k$ and all $\tau\in \Sym(k)$. \footnote{The concept of an invariant partition has already been introduced in Theorem 6.1 in \cite{pono}.}
\end{definition}

Fix some $r\in [k]$ and $\vec{i}\in [k]^r$.  For a tuple $\vec{v} \in V^k$, we define its \emph{projection} on $\vec{i}$, denoted  $\pr_{\vec{i}}\vec{v}$ to be the tuple in $V^r$  with $j^{th}$ entry $v_{i_j}$.  Without ambiguity, we write $\pr_{r}\vec{v}$ for the tuple $\pr_{(1,\hdots,r)}\vec{v}$.
We denote by $\pr_{r}\gamma$ the partition of $V^r$ given by
$$
\pr_{r}\gamma(\vec{v})=\gamma(v_{1},v_{2},\hdots,v_{r},v_{r},\hdots,v_{r})
$$
and call it the $r$-\emph{projection} of $\gamma$. Note that if $\gamma$ is invariant, then for any $r \leq k$, $\pr_{r}\gamma$ is invariant.

\begin{definition}[$r$-consistency]
$\gamma$ is said to be $\vec{i}$-\emph{consistent} for some $\vec{i}\in [k]^r$ if for all $\vec{u},\vec{v}\in V^k$
$$
\gamma(\vec{u})=\gamma(\vec{v})\implies \pr_{r}\gamma(\pr_{\vec{i}}\vec{u})=\pr_{r}\gamma(\pr_{\vec{i}}\vec{v})
$$
If, in addition, $\gamma$ is $\vec{i}$-consistent for all $\vec{i}\in [k]^r$ we say that $\gamma$ is $r$-\emph{consistent}.
\end{definition}
Observe that if for some $r\leq k$, $\gamma$ is $r$-consistent, then it is $t$-consistent for all $t\leq r$. One may also verify that if $\gamma$ is invariant, then it is $k$-consistent if and only if for all $\vec{u},\vec{v}\in V^k$
\begin{equation}\label{consist}
\gamma(\pr_{k-1}\vec{u}\cdot u_{k-1})=\gamma(\pr_{k-1}\vec{v}\cdot v_{k-1}).
\end{equation}
\begin{definition}[Graph-like partition]
$\gamma$ is said to be a \emph{graph-like} partition of $V^k$ if it is invariant, $r$-consistent for all $r\leq k$ and for all $\vec{u},\vec{v}\in V^k$
\begin{equation}\label{prop}
\gamma(\vec{u})=\gamma(\vec{v})\implies(u_i=u_j\implies v_i=v_j\,\,\forall i,j\in [k]).
\end{equation}
\end{definition}
Note that if $\gamma$ is graph-like, then $\pr_t\gamma$ is graph-like for all $t\in [k]$. An example of graph-like partition is that of a coherent configuration, as introduced in Section \ref{cohs}.
\begin{proposition}\label{gp1}
A coherent configuration is a graph-like partition.
\end{proposition}
\begin{proof}
Let $\rho$ be a coherent configuration on $V$. Then $\rho$ is a rainbow, and hence satisfies conditions $(1)$ and $(2)$ in Definition \ref{cohrain}, from which we deduce that it is invariant and satisfies formula (\ref{prop}). To show that it is $1$-consistent, let $u,v,u',v'\in V$ be such that $\rho(u,v)=\rho(u',v')$. For any $\sigma\in \im(\rho)$
$$
\{x\in V \mid \rho(u,x)=\sigma\}=\bigcup_{\tau\in\im(\rho)} \{x\in V \mid \rho(u,x)=\sigma,\rho(x,v)=\tau\}.
$$
The size of the right-hand side of the above  is independent of the choice of $(u,v)$ from the equivalence class $[(u,v)]_\gamma$. In particular, if $\sigma=\rho(u,u)$, then because $\rho(u,v)=\rho(u',v')$ there is exactly one $x\in V$ such that $\rho(u',x)=\sigma$, namely $x=u'$. Hence, $\rho(u,u)=\rho(u',u')$ and $1$-consistency of $\rho$ follows.
\end{proof}
Arguments of this kind appear repeatedly in our proofs of Theorems $\ref{thm:main1}$, $\ref{thm:main2}$ and $\ref{thm:main3}$.

Another graph-like partition which will be useful throughout the paper is that of \emph{atomic types} of $k$-tuples of vertices of a graph $\Gamma$, which we indicate by $\alpha_{k,\Gamma}$.  To be precise, we define
$$
\begin{matrix}
\alpha_{k,\Gamma}:& V^k& \rightarrow & \im(\Gamma)^{[k]^{(2)}}\\
& \vec{v} & \mapsto & (\Gamma(v_i,v_j))_{(i,j)\in[k]^{(2)}},
\end{matrix}
$$
where, as the reader may recall, we view graphs as labelled partitions.
\begin{definition}[$\Gamma$-partition]
We say that $\gamma$ is a $\Gamma$-partition if $\alpha_{k,\Gamma}\preceq \gamma$ for some graph $\Gamma$.
\end{definition}

The following result is a useful property of graph-like partitions.
\begin{lemma}\label{veryuseful}
For any graph-like $\gamma\in\mathcal{P}(V^k)$ and any $\vec{u},\vec{v}\in V^k$ such that $\gamma(\vec{u})=\gamma(\vec{v})$
$$
\gamma(\vec{u}\langle \vec{i},\pr_{\vec{j}}\vec{u}\rangle)=\gamma(\vec{v}\langle \vec{i},\pr_{\vec{j}}\vec{v}\rangle)
$$
for all $\vec{i}\in [k]^{(r)}$ and $\vec{j}\in [k]^r$.
\end{lemma}
\begin{proof}
Let $\vec{u},\vec{v}$ be as in the statement. For any $s,t\in [k]$, $\gamma(\vec{u}\langle s,v_t\rangle)=\gamma(\vec{v}\langle s,u_t\rangle)$ holds true by invariance and $(k-1)$-consistency of $\gamma$.

Fix some $r<k$ and set $\vec{u}'=\pr_{\vec{j}}\vec{u}$ and $\vec{v}'=\pr_{\vec{j}}\vec{v}$ for some $\vec{j}\in [k]^r$. Suppose that $\gamma(\vec{v}\langle \vec{i},\vec{v}'\rangle)=\gamma(\vec{u}\langle \vec{i},\vec{u}'\rangle)$ holds for all $\vec{i}\in [k]^{(r)}$. Because $r<k$, there is some $i'\in [k]$ which is not an entry of $\vec{i}$. For some $l\in [k]$, let $x=v_l$ and $y=u_l$. Then
$$
\gamma((\vec{v}\langle \vec{i},\vec{v}'\rangle)\langle i',x\rangle)=\gamma((\vec{u}\langle \vec{i},\vec{u}'\rangle)\langle i',y\rangle)
$$
by invariance and $(k-1)$-consistency of $\gamma$, from which one concludes that
$$
\gamma(\vec{v}\langle \vec{i}\cdot i',\vec{v}'\cdot x\rangle)=\gamma(\vec{u}\langle \vec{i}\cdot i',\vec{u}'\cdot y\rangle)
$$
and hence, by setting $\vec{q}=\vec{j}\cdot l$
$$
\gamma(\vec{v}\langle \vec{i}\cdot i',\pr_{\vec{q}}\vec{v}\rangle)=\gamma(\vec{u}\langle \vec{i}\cdot i',\pr_{\vec{q}}\vec{u}\rangle).
$$
We deduce the desired statement by induction.
\end{proof}

\section{Refinement operators and procedures}\label{rop}

As previously stated, all the SPAS in this paper arise from \emph{refinement operators}, which we now define.  Fix some $k\in\mathbb{N}$.

\begin{definition}[Refinement operator]
A $k$-\emph{refinement operator} $R$ is a mapping which, for each set $V$,  assigns  to each $\gamma\in\mathcal{P}(V^k)$ a partition $R\circ\gamma\in \mathcal{P}(V^k)$ such that $\gamma\preceq R\circ\gamma$.
\end{definition}
Since labelled partitions are seen as mappings, the symbol $\circ$ really indicates composition thereof.
\begin{definition}[Fixed point]
$\gamma\in\mathcal{P}(V^k)$ is said to be a \emph{fixed point} of a $k$-refinement operator $R$ if $\gamma\approx R\circ\gamma$. In such cases we also say that $\gamma$ is $R$-\emph{stable}.
\end{definition}

Fix some $\gamma\in\mathcal{P}(V^k)$, set $X^0=\gamma$ and $X^i=R\circ X^{i-1}$. We then have an increasing sequence:
$$
X^0\preceq X^1\preceq \hdots \preceq X^i \preceq\hdots
$$

Because all elements of this sequence are bounded by a labelled partition of $V^k$ with exactly one element per equivalence class, there must be some $s\in\nats$ such that for all $i\geq s$, $X^i$ is  fixed point of $R$. For the smallest such $s$, denote $X^s$ by $[\gamma]^R$.

\begin{definition}[Graph-like operator]
A $k$-refinement operator $R$ is \emph{graph-like} if $R\circ\gamma$ is graph-like for all graph-like $\gamma\in \mathcal{P}(V^k)$ and sets $V$.
\end{definition}

\begin{definition}[Refinement procedure]\label{repr}
The family of mappings $\{R_1,R_2,\hdots\}$ is said to be a \emph{refinement procedure} if for each $k\in \mathbb{N}$
\begin{enumerate}
\item $R_k$ is a graph-like $k$-refinement operator.
\item If $\gamma$ is a graph-like fixed point of $R_k$ then $\pr_{k-1}\gamma$ is a fixed point of $R_{k-1}$.
\item For all sets $V$ and $\gamma\in\mathcal{P}(V^k)$, $[\gamma]^{R_k}$ is computable in time $|V|^{O(k)}$.
\end{enumerate}
\end{definition}

For each $k,r\in \nats$, the $k$-refinement operators of interest in this paper are the \emph{Weisfeiler-Leman operators} $\WL_{k,r}$, the \emph{counting logic operators} $\C_{k,r}$ and, for any field $\mathbb{F}$, the \emph{invertible map operators} $\IM_k^{\mathbb{F}}$.  For the former two, we are really interested in the case when $r <k$.  Thus, when  $k\leq r$, for convenience, we let $\WL_{k,r}\circ\gamma=\C_{k,r}\circ\gamma=\gamma$ for all $\gamma\in\mathcal{P}(V^k)$ and sets $V$. For $r < k$ define
$$
\begin{matrix}
\WL_{k,r}\circ \gamma:  &V^k&\rightarrow &\im(\gamma)\times\mult(\Phi^{\gamma,r})^{V^r}\\
 &\vec{v} &\mapsto  & (\gamma(\vec{v}),\{\{(\gamma(\vec{v}\langle \vec{i},\vec{u}\rangle)_{\vec{i}\in [k]^{(r)}} \mid \vec{u}\in V^r\}\}).
\end{matrix}
$$
$$
\begin{matrix}
\C_{k,r}\circ\gamma: & V^k& \rightarrow & \im(\gamma)\times(\mult(\im(\gamma)^{V^r}))^{[k]^{(r)}}\\
& \vec{v} &\mapsto & (\gamma(\vec{v}),(\{\{\gamma(\vec{v}\langle \vec{i},\vec{u}\rangle) \mid \vec{u}\in V^r\}\})_{\vec{i}\in [k]^{(r)}}).
\end{matrix}
$$
Let $\chi_{\vec{i},\sigma}^{\gamma,\vec{v}}$ be the adjacency matrix of the binary relation $\{(x,y)\in V^2 \mid \gamma(\vec{v}\langle \vec{i},(x,y)\rangle)=\sigma\}$. Similarly to the above, set $\IM_1^\mathbb{F}\circ\gamma=\IM_2^\mathbb{F}\circ\gamma=\gamma$ and, for $k>2$ define
$$
\begin{matrix}
\IM_k^{\mathbb{F}}\circ\gamma: & V^k& \rightarrow & \im(\gamma)\times(\mathrm{Mat}_V(\mathbb{F})^{\im(\gamma)\times [k]^{(2)}}/\sim)\\
& \vec{v}&\mapsto & (\gamma(\vec{v}),((\chi^{\gamma,\vec{v}}_{\vec{i},\sigma})_{\sigma\in \im(\gamma)})_{\vec{i}\in [k]^{(2)}})
\end{matrix}
$$
where $\sim$ is the equivalence relation on elements of $\mathrm{Mat}_V(\mathbb{F})^{\im(\gamma)\times [k]^{(2)}}$ under simultaneous similarity. That is, two tuples are equivalent if they lie in the same orbit of $GL_V(\mathbb{F})$ acting on the tuples by conjugation. Although the reader may find the above definitions rather technical, they are not crucial throughout the paper. Indeed, for $k>r$ and $\vec{u},\vec{v}\in V^k$ the following facts are sufficient:
\begin{enumerate}
\item $\WL_{k,r}\circ\gamma(\vec{u})=\WL_{k,r}\circ\gamma(\vec{v})$ if, and only if, $\gamma(\vec{u})=\gamma(\vec{v})$ and for all $\vec{\phi}\in \Phi^{\gamma,r}$ and $\vec{i}\in [k]^{(r)}$
$$
|\{\vec{x}\in V^r \mid \gamma(\vec{u}\langle \vec{i},\vec{x}\rangle)=\phi_{\vec{i}}\}|=|\{\vec{x}\in V^r \mid \gamma(\vec{v}\langle \vec{i},\vec{x}\rangle)=\phi_{\vec{i}}\}|.
$$
\item $\C_{k,r}\circ\gamma(\vec{u})=\C_{k,r}\circ\gamma(\vec{v})$ if $\gamma(\vec{u})=\gamma(\vec{v})$ and for all $\sigma\in\im(\gamma)$ and $\vec{i}\in [k]^{(r)}$
$$
|\{\vec{x}\in V^r \mid \gamma(\vec{u}\langle \vec{i},\vec{x}\rangle)=\sigma\}|=|\{\vec{x}\in V^r \mid \gamma(\vec{v}\langle \vec{i},\vec{x}\rangle)=\sigma\}|.
$$
\item $\IM_k^{\mathbb{F}}\circ\gamma(\vec{u})=\IM_k^{\mathbb{F}}\circ\gamma(\vec{v})$ if, and only if, $\gamma(\vec{u})=\gamma(\vec{v})$ and for each $\vec{i}\in[k]^{(2)}$ there exist some $S\in \mathrm{GL}_V(\mathbb{F})$ such that for all $\sigma\in\im(\gamma)$
$$
S\chi_{\vec{i},\sigma}^{\gamma,\vec{u}}S^{-1}=\chi_{\vec{i},\sigma}^{\gamma,\vec{v}}.
$$
\end{enumerate}

From this, one may derive the following \emph{stability conditions}.
\begin{proposition}\label{stab}
For any $\gamma\in\mathcal{P}(V^k)$ and $k>r$:
\begin{enumerate}
\item $\gamma$ is $\WL_{k,r}$-stable if, and only if, for all $\vec{v}\in V^k,\vec{i}\in [k]^{(r)}$ and $\vec{\phi}\in \Phi^{\gamma,r}$, the size of the set $\{\vec{x}\in V^r \mid \gamma(\vec{v}\langle \vec{i},\vec{x}\rangle)=\phi_{\vec{i}}\}$ is independent of the choice of $\vec{v}$ from the equivalence class $[\vec{v}]_\gamma$.
\item $\gamma$ is $\C_{k,r}$-stable if, and only if, for all $\vec{v}\in V^k,\vec{i}\in [k]^{(r)}$ and $\sigma\in \im(\gamma)$, the size of the set $\{\vec{x}\in V^r \mid \gamma(\vec{v}\langle \vec{i},\vec{x}\rangle)=\sigma\}$ is independent of the choice of $\vec{v}$ from the equivalence class $[\vec{v}]_\gamma$.
\item $\gamma$ is $\IM_k^{\mathbb{F}}$-stable if for all $\vec{u},\vec{v}\in V^k$ and $\vec{i}\in[k]^{(2)}$
$$
\gamma(\vec{u})=\gamma(\vec{v})\implies \exists S\in \mathrm{GL}_V(\mathbb{F}),S\chi_{\vec{i},\sigma}^{\gamma,\vec{u}}S^{-1}=\chi_{\vec{i}\,\sigma}^{\gamma,\vec{v}}\,\,\forall \sigma\in\im(\gamma).
$$
\end{enumerate}
\end{proposition}

Consider the following families of mappings:
\begin{enumerate}
\item For all $r\in \nats$, $\mathbf{WL}_r=\{\WL_{1,r},\WL_{2,r},\hdots\}$.
\item For all $r\in \nats$, $\mathbf{C}_r=\{\C_{1,r},\C_{2,r},\hdots\}$.
\item For any field $\mathbb{F}$, $\mathbf{IM}(\mathbb{F})=\{\IM_1^\mathbb{F},\IM_2^{\mathbb{F}},\hdots\}$.
\end{enumerate}

\begin{proposition}\label{ug}
The families $\mathbf{WL}_r$, $\mathbf{C}_r$ and $\mathbf{IM}(\mathbb{F})$ are refinement procedures for all $r\in \nats$ and fields $\mathbb{F}$.
\end{proposition}

For the proof of Proposition~\ref{ug} and that of the next auxiliary Lemma, we use the following notations and conventions.  Fix a graph-like partition $\gamma\in\mathcal{P}(V^k)$.  Let $\overline\gamma=\pr_{k-1}\gamma$ and for all $\vec{v}\in V^k$, if $\sigma=\gamma(\vec{v})$, let $\overline\sigma=\overline\gamma(\pr_{k-1}\vec{v})$.  For $\pi\in\Sym(k)$, we let $\sigma^\pi=\gamma(\vec{v}^\pi)$, and define an action on $[k]^{(r)}$ by setting $(\pi(\vec{i}))_j=\pi(i_j)$ for all $j\in[r]$. For any $\vec{v}\in V^k$ we denote $\vec{v}'=\pr_{k-1}\vec{v}\cdot v_{k-1}$. Note that $\overline\sigma$ and $\sigma^\pi$ are well defined, since $\gamma$ is graph-like.
\begin{lemma}\label{useful}
Let $\vec{i}\in[k]^{(r)}$. Then:
\begin{enumerate}
\item If $k$ is an entry of $\vec{i}$, then $\gamma(\vec{v}'\langle\vec{i},\vec{x}\rangle)=\gamma(\vec{v}\langle \vec{i},\vec{x}\rangle)$.
\item If $k-1$ is an entry of $\vec{i}$ but $k$ is not, then $\gamma(\vec{v}^\pi\langle\vec{i},\vec{x}\rangle)=\gamma(\vec{v}'\langle \vec{i},\vec{x}\rangle)$, where $\pi=(k-1,\;k)$ is a transposition of $\Sym(k)$.
\item If neither $k$ nor $k-1$ are entries of $\vec{i}$, then $\gamma(\vec{v}\langle\vec{i},\vec{x}\rangle)=\sigma\implies \gamma(\vec{v}'\langle\vec{i},\vec{x}\rangle)=\overline\sigma$. In particular,
\begin{equation}\label{ro1}
\{\vec{x}\in V^r \mid \gamma(\vec{v}'\langle \vec{i},\vec{x}\rangle)=\kappa\}=\bigcup_{\{\sigma\in\im(\gamma) \mid \overline\sigma=\kappa\}}\{\vec{x}\in V^r \mid \gamma(\vec{v}\langle\vec{i},\vec{x}\rangle)=\sigma\}.
\end{equation}
\end{enumerate}
\end{lemma}
\begin{proof}
Statements $(1)$ and $(2)$ are trivial to check.

Let $\vec{w}=\vec{v}\langle\vec{i},\vec{x}\rangle$ with $\gamma(\vec{w})=\sigma$. Then $\vec{w}'=\vec{v}'\langle \vec{i},\vec{x}\rangle$, so $\gamma(\vec{w}')=\overline\sigma$ by definition. From this, and the fact that $\gamma$ is $(k-1)$-consistent, the right-hand side of formula (\ref{ro1}) is a subset of the left-hand side. The reverse inclusion follows from the definition of $\overline\sigma$ in terms of $\sigma$, and statement $(3)$ follows.
\end{proof}

\begin{proof}[Proof of Proposition \ref{ug}]
We check that each of the families of mappings satisfy the conditions in Definition~\ref{repr}. Note that $\gamma\preceq R\circ\gamma$ for any $k$-refinement operator $R$. Hence, since $\gamma$ is graph-like, for any $\vec{u},\vec{v}\in V^k$, $R\circ\gamma(\vec{u})=R\circ\gamma(\vec{v})$ implies that $u_i=u_j\implies v_i=v_j$ for all $i,j\in [k]$. Thus, to show that $R\circ\gamma$ is graph-like, it suffices to verify that $R\circ \gamma$ is invariant and satisfies (\ref{consist}). That is, that  for all $\vec{u},\vec{v}\in V^k$, $R\circ\gamma(\vec{u})=R\circ\gamma(\vec{v})\implies R\circ\gamma(\vec{u}')=R\circ\gamma(\vec{v}').$
\item
\paragraph{$\mathbf{WL}_r$ is a refinement procedure.}
We first show that $\WL_{k,r}\circ\gamma$ is invariant and satisfies formula (\ref{consist}), and is thus a graph-like partition.

Suppose $\WL_{k,r}\circ\gamma(\vec{u})=\WL_{k,r}\circ\gamma(\vec{v})$. Then $\gamma(\vec{u})=\gamma(\vec{v})$ by definition, and hence $\gamma(\vec{u}^\tau)=\gamma(\vec{v}^\tau)$ since $\gamma$ is graph-like and thus invariant. Furthermore, from the invariance of $\gamma$ it follows that for all $\tau\in\Sym(k)$ and $\vec{\phi}\in\Phi^{\gamma,r}$
$$
\{\vec{x}\in V^r \mid \gamma(\vec{v}^\tau\langle \tau(\vec{i}),\vec{x}\rangle)=(\phi_{\vec{i}})^\tau\,\,\forall\vec{i}\in [k]^{(r)}\}=\{\vec{x}\in V^r \mid \gamma(\vec{v}\langle \vec{i},\vec{x}\rangle)=\phi_{\vec{i}}\,\forall\vec{i}\in [k]^{(r)}\}.
$$
Hence, $\WL_{k,r}\circ\gamma$ is invariant.

For all $\vec{\phi}\in \Phi^{\gamma,r}$ let $\vec{\phi}^\dag\in\Phi^{\gamma,r}$ be defined as
$$
\phi^\dag_{\vec{i}}=\begin{cases}
\overline\phi_{\vec{i}}\;\;\textrm{ if neither $k$ nor $k-1$ are entries of $\vec{i}$}\\
(\phi_{\vec{i}})^\pi\;\;\textrm{ if $k-1$ is an entry of $\vec{i}$}\\
\phi_{\vec{i}}\;\;\textrm{ otherwise }
\end{cases}
$$
where $\pi=(k-1,k)\in \Sym(k)$. Observe that if $\{\vec{x}\in V^r \mid \gamma(\vec{v}'\langle \vec{i},\vec{x}\rangle)=\psi_{\vec{i}}\,\,\forall\vec{i}\in[k]^{(r)}\}$ is non-empty for some $\vec{v}\in V^k$ and $\vec{\psi}\in \Phi^{\gamma,r}$, then $\vec{\psi}=\vec{\phi^\dag}$ for some $\vec{\phi}\in \Phi^{\gamma,r}$. It follows from the definition of $\vec{\phi^\dag}$ and Lemma \ref{useful} that
\begin{equation}\label{ro2}
\{\vec{x}\in V^r \mid \gamma(\vec{v}'\langle \vec{i},\vec{x}\rangle)=\psi_{\vec{i}}\,\,\forall\vec{i}\in [k]^{(r)}\}=\bigcup_{\{\vec{\phi}\in\Phi^{\gamma,r} \mid \vec{\psi}=\vec{\phi^\dag}\}}\{\vec{x}\in V^r \mid \gamma(\vec{v}\langle \vec{i},\vec{x}\rangle)=\phi_{\vec{i}}\,\,\forall\vec{i}\in[k]^{(r)}\}.
\end{equation}
Since $\WL_{k,r}\circ\gamma(\vec{u})=\WL_{k,r}\circ\gamma(\vec{v})$, for all $\vec{\phi}\in \Phi^{\gamma,r}$
$$
|\{\vec{x}\in V^r \mid \gamma(\vec{u}\langle \vec{i},\vec{x}\rangle)=\phi_{\vec{i}}\,\,\forall\vec{i}\in[k]^{(r)}\}|=|\{\vec{x}\in V^r \mid \gamma(\vec{v}\langle \vec{i},\vec{x}\rangle)=\phi_{\vec{i}}\,\,\forall\vec{i}\in[k]^{(r)}\}|.
$$
Thus, because the right-hand side of formula (\ref{ro2}) is a disjoint union, we deduce that
$$
|\{\vec{x}\in V^r \mid \gamma(\vec{u}'\langle \vec{i},\vec{x}\rangle)=\phi_{\vec{i}}\,\,\forall\vec{i}\in [k]^{(r)}\}|=|\{\vec{x}\in V^r \mid \gamma(\vec{v}'\langle \vec{i},\vec{x}\rangle)=\phi_{\vec{i}}\,\,\forall\vec{i}\in [k]^{(r)}\}|
$$
for all $\vec{\phi}\in \Phi^{\gamma,r}$, which implies $\WL_{k,r}\circ\gamma(\vec{u}')=\WL_{k,r}\circ\gamma(\vec{v}')$, and hence, that $\WL_{k,r}\circ\gamma$ is graph-like.

Suppose now that $\gamma$ is $\WL_{k,r}$-stable. Observe that for all $\vec{\phi}\in\Phi^{\overline\gamma,r}$
\begin{equation}\label{ro3}
\{\vec{x}\in V^r \mid \overline\gamma(\pr_{k-1}\vec{v}\langle\vec{i},\vec{x}\rangle)=\phi_{\vec{i}}\forall\vec{i}\in [k-1]^{(r)}\}=\bigcup_{\{\vec{\psi}\in\Phi^{\gamma,r} \mid \overline{\psi}_{\vec{i}}=\phi_{\vec{i}}\forall\vec{i}\in [k-1]^{(r)}\}}\{\vec{x}\in V^r \mid \gamma(\vec{v}\langle\vec{i},\vec{x}\rangle)=\psi_{\vec{i}}\forall\vec{i}\in [k]^{(r)}\}.
\end{equation}
From the $\WL_{k,r}$-stability of $\gamma$, and the fact that $\gamma(\vec{u})=\gamma(\vec{v})$, it follows that for all $\vec{\psi}\in \Phi^{\gamma,r}$
$$
|\{\vec{x}\in V^r \mid \gamma(\vec{u}\langle\vec{i},\vec{x}\rangle)=\psi_{\vec{i}}\forall\vec{i}\in [k]^{(r)}\}|=|\{\vec{x}\in V^r \mid \gamma(\vec{v}\langle\vec{i},\vec{x}\rangle)=\psi_{\vec{i}}\forall\vec{i}\in [k]^{(r)}\}|.
$$
Since the right-hand side of formula (\ref{ro3}) is a disjoint union, it holds that for all $\vec{\phi}\in \Phi^{\overline\gamma,r}$
$$
|\{\vec{x}\in V^r \mid \overline\gamma(\pr_{k-1}\vec{u}\langle\vec{i},\vec{x}\rangle)=\phi_{\vec{i}}\forall\vec{i}\in [k-1]^{(r)}\}|=|\{\vec{x}\in V^r \mid \overline\gamma(\pr_{k-1}\vec{v}\langle\vec{i},\vec{x}\rangle)=\phi_{\vec{i}}\forall\vec{i}\in [k-1]^{(r)}\}|.
$$
Thus, $\overline\gamma$ is $\WL_{k-1,r}$-stable.

To see that the stable partition $[\gamma]^{\WL_{k,r}}$ can be computed in time $|V|^{O(k)}$, note that the number of iterations of the refinement operator before the fixed point is reached is at most $|V|^k$.  At each step, for each $k$-tuple $\vec{v}$, we need to compute the colour $\WL_{k,r}\circ\gamma(\vec{v})$ which involves checking the colour of $k|V|$ distinct tuples.  The total time required at each step is therefore $O(k|V|^{k+1})$.  Repeating this for $|V|^k$ steps gives us the required bound.
\item
\paragraph{$\mathbf{C}_{r}$ is a refinement procedure.}
The proof in this case, is similar to that of $\WL_{k,r}$.

Suppose $\C_{k,r}\circ\gamma(\vec{u})=\C_{k,r}\circ\gamma(\vec{v})$. The invariance of $\gamma$ implies that $\gamma(\vec{u}^\tau)=\gamma(\vec{v}^\tau)$, and that
$$
\{\vec{x}\in V^r \mid \gamma(\vec{v}^\tau\langle \tau(\vec{i}),\vec{x}\rangle)=\sigma^\tau\}=\{\vec{x}\in V^r \mid \gamma(\vec{v}\langle \vec{i},\vec{x}\rangle)=\sigma\}
$$
for all $\tau\in\Sym(k)$ and $\vec{i}\in [k]^{(r)}$. Hence
$$
|\{\vec{x}\in V^r \mid \gamma(\vec{u}^\tau\langle \tau(\vec{i}),\vec{x}\rangle)=\sigma^\tau\}|=|\{\vec{x}\in V^r \mid \gamma(\vec{v}^\tau\langle \tau(\vec{i}),\vec{x}\rangle)=\sigma^\tau\}|
$$
and therefore $\C_{k,r}\circ\gamma(\vec{u}^\tau)=\C_{k,r}\circ\gamma(\vec{v}^\tau)$, thus showing that $\C_{k,r}\circ\gamma$ is invariant.

As $\gamma$ is graph-like, and $\C_{k,r}\circ\gamma(\vec{u})=\C_{k,r}\circ\gamma(\vec{v})$, it follows from $(1)$ in Lemma \ref{useful} that for all $\vec{i}\in [k]^{(r)}$ with some entry equal to $k$
\begin{equation}\label{useful8}
|\{\vec{x}\in V^r \mid \gamma(\vec{u}'\langle \vec{i},\vec{x}\rangle)=\sigma\}|=|\{\vec{x}\in V^r \mid \gamma(\vec{v}'\langle \vec{i},\vec{x}\rangle)=\sigma\}|.
\end{equation}
If $\vec{i}\in [k]^{(r)}$ has an entry equal $k-1$ but none equal to $k$, then (2) in Lemma \ref{useful} implies that 
$$
|\{\vec{x}\in V^r \mid \gamma(\vec{v}'\langle \vec{i},\vec{x}\rangle)=\sigma\}|=|\{\vec{x}\in V^r \mid \gamma(\vec{v}^\pi\langle \pi(\vec{i}),\vec{x}\rangle)=\sigma\}|.
$$
From the invariance of $\C_{k,r}\circ\gamma$ we deduce that formula (\ref{useful8}) holds also for such values of $\vec{i}$. Finally, if no entry of $\vec{i}$ is equal to $k$ or $k-1$, $\gamma$ satisfies formula (\ref{ro1}). As the right-hand side of the latter is a disjoint union, and for all $\sigma\in\im(\gamma)$
$$
|\{\vec{x}\in V^r \mid \gamma(\vec{u}\langle\vec{i},\vec{x}\rangle)=\sigma\}|=|\{\vec{x}\in V^r \mid \gamma(\vec{v}\langle\vec{i},\vec{x}\rangle)=\sigma\}|
$$
it follows that formula (\ref{useful8}) holds also when neither $k$ nor $k-1$ are entries of $\vec{i}$. Thus, $C_{k,r}\circ\gamma(\vec{u}')=\C_{k,r}\circ\gamma(\vec{v}')$ and $\C_{k,r}\circ\gamma$ is therefore graph-like.

Suppose $\gamma$ is $\C_{k,r}$-stable. Since $\gamma$ is graph-like, it holds that
\begin{equation}\label{rol4}
\{\vec{x}\in V^r \mid \overline\gamma(\pr_{k-1}\vec{v}\langle\vec{i},\vec{x}\rangle)=\kappa\}=\bigcup_{\{\sigma\in\im(\gamma) \mid \overline{\sigma}=\kappa\}}\{\vec{x}\in V^r \mid \gamma(\vec{v}\langle\vec{i},\vec{x}\rangle)=\sigma\}.
\end{equation}
for all $\vec{i}\in [k-1]^{(r)}$ and $\kappa\in\im(\overline\gamma)$. Since the right-hand side of the above is a disjoint union and
$$
|\{\vec{x}\in V^r \mid \gamma(\vec{u}\langle \vec{i},\vec{x}\rangle)=\sigma\}|=|\{\vec{x}\in V^r \mid \gamma(\vec{v}\langle \vec{i},\vec{x}\rangle)=\sigma\}|
$$
it then follows from formula (\ref{rol4}) that 
$$
|\{\vec{x}\in V^r \mid \overline\gamma(\pr_{k-1}\vec{u}\langle \vec{i},\vec{x}\rangle)=\sigma\}|=|\{\vec{x}\in V^r \mid \overline\gamma(\pr_{k-1}\vec{v}\langle \vec{i},\vec{x}\rangle)=\sigma\}|.
$$
Thus, $\overline\gamma$ is $\C_{k-1,r}$-stable.

A very similar argument to the case of $\WL_{k,r}$ shows that $[\gamma]^{\C_{k,r}}$ can be computed in time $|V|^{O(k)}$.
\item
\paragraph{$\mathbf{IM}(\mathbb{F})$ is a refinement procedure.}
We proceed via the same proof strategy as for the above refinement procedures.

Suppose $\IM_k^\mathbb{F}\circ\gamma(\vec{u})=\IM_k^\mathbb{F}\circ\gamma(\vec{v})$. Then $\gamma(\vec{u})=\gamma(\vec{v})$ and hence, $\gamma(\vec{u}^\tau)=\gamma(\vec{v}^\tau)$ for all $\tau\in\Sym(k)$. Note that for any $\sigma\in\im(\gamma)$, $\chi^{\gamma,\vec{v}}_{\vec{i},\sigma}=\chi^{\gamma,\vec{v}^\tau}_{\tau(\vec{i}),\sigma^\tau}$, from which one deduces that $\IM_k^\mathbb{F}\circ\gamma(\vec{u}^\tau)=\IM_k^\mathbb{F}\circ\gamma(\vec{v}^\tau)$, whence, $\IM_k^{\mathbb{F}}\circ\gamma$ is invariant.

From Lemma \ref{useful}, it follows that $\chi^{\gamma,\vec{v}'}_{\vec{i},\sigma}=\chi^{\gamma,\vec{v}}_{\vec{i},\sigma}$ if $k$ is an entry of $\vec{i}$, $\chi^{\gamma,\vec{v}'}_{\vec{i},\sigma}=\chi^{\gamma,\vec{v}^\pi}_{\pi(\vec{i}),\sigma}$ if $k-1$ is an entry of $\vec{i}$, but $k$ is not, and
$$
\chi^{\gamma,\vec{v}'}_{\vec{i},\sigma}=\sum_{\{\kappa\in\im(\gamma) \mid \overline\kappa=\sigma\}}\chi^{\gamma,\vec{v}}_{\vec{i},\kappa}
$$
if neither $k$ nor $k-1$ are entries of $\vec{i}$. Since $S\chi^{\gamma,\vec{v}}_{\vec{i},\sigma}S^{-1}=\chi^{\gamma,\vec{u}}_{\vec{i},\sigma}$ for all $\sigma\in\im(\gamma)$, then $S\chi^{\gamma,\vec{v}'}_{\vec{i},\sigma}S^{-1}=\chi^{\gamma,\vec{u}'}_{\vec{i},\sigma}$ for all $\sigma\in\im(\gamma)$. From this, $\IM_k^{\mathbb{F}}\circ\gamma(\vec{u}')=\IM_k^{\mathbb{F}}\circ\gamma(\vec{v}')$ follows, and thus, $\IM_k^{\mathbb{F}}\circ\gamma$ is graph-like.

From the fact that $\gamma$ is graph-like we deduce that for all $\vec{i}\in [k-1]^{(2)}$
$$
\chi^{\overline\gamma,\pr_{k-1}\vec{v}}_{\vec{i},\sigma}=\sum_{\{\kappa\in\im(\gamma) \mid \overline{\kappa}=\sigma\}}\chi_{\vec{i},\kappa}^{\gamma,\vec{v}}.
$$
Thus, if $S\chi^{\gamma,\vec{v}}_{\vec{i},\sigma}S^{-1}=\chi^{\gamma,\vec{u}}_{\vec{i},\sigma}$ for all $\sigma\in\im(\gamma)$, then $S\chi^{\overline\gamma,\pr_{k-1}\vec{v}}_{\vec{i},\sigma}S^{-1}=\chi^{\overline\gamma,\pr_{k-1}\vec{u}}_{\vec{i},\sigma}$ for all $\sigma\in\im(\gamma)$. In particular, if $\gamma$ is $\IM_k^\mathbb{F}$-stable, then $\overline\gamma$ is $\IM_{k-1}^\mathbb{F}$-stable.

Finally, showing that $[\gamma]^{\IM_k^\mathbb{F}}$ can be computed in time $|V|^{O(k)}$ is similar to the previous cases.  Again, the number of refinement steps is at most $|V|^k$.  However, at each step, and for each pair of tuples, we need to perform a simultaneous similarity test.  For this, we rely on the fact that simultaneous similarity is decidable in polynomial time.  This follows from the fact that testing simultaneous similarity can be reduced in polynomial time to testing module isomorphism (see~\cite{chistov}, for example) and the polynomial-time algorithm for the latter problem over any field is given by Brookbanks and Luks in~\cite{luks}.
\end{proof}
For a refinement procedure $\mathbf{R}=\{R_1,R_2,\hdots\}$, let $\mathcal{S}_{\mathbf{R}}=\{\overline{R}_1,\overline{R}_2,\hdots\}$ be the family of mappings where for all graphs $\Gamma$, $\overline{R}_1(\Gamma)=\Gamma$ and $\overline{R}_k(\Gamma)=\pr_2[\alpha_{k,\Gamma}]^{R_k}$ for $k\geq2$. Then, from the above discussion, $\mathcal{S}_{\mathbf{R}}$ is a SPAS. We define $\mathcal{S}_{\mathbf{WL},r}=\{\overline{\WL}_{1,r},\overline{\WL}_{2,r},\hdots\}, \mathcal{S}_{\mathbf{C},r}=\{\overline{\C}_{1,r},\overline{\C}_{2,r},\hdots\}$ and $\mathcal{S}_{\mathbf{IM}(\mathbb{F})}=\{\overline{\IM}_1^\mathbb{F},\overline{\IM}_2^\mathbb{F},\hdots\}$ to be the SPAS' obtained in such a manner from the refinement procedures $\mathbf{WL}_r, \mathbf{C}_r$ and $\mathbf{IM}(\mathbb{F})$ respectively.

\section{Proof of Theorem $\ref{thm:main1}$}

Hereafter, when talking of the families of operators $\mathbf{WL}_r$ and  $\mathbf{C}_r$ with $r=1$, we drop the subscript.  That is, we write  $\mathbf{WL}=\{\WL_1,\WL_2,\hdots\}$ and $\mathbf{C}=\{\C_1,\C_2,\hdots\}$ to denote the families $\mathbf{WL}_1=\{\WL_{1,1},\WL_{2,1},\hdots\}$ and $\mathbf{C}_1=\{\C_{1,1},\C_{2,1},\hdots\}$ respectively.  The distinction between the \emph{procedure} $\mathbf{WL}_k$ and the \emph{operator} $\WL_k$ should not cause any confusion, likewise for the procedure $\mathbf{C}_k$ and the operator $\C_k$.

For showing that for two refinement procedures $\mathbf{X}=\{X_1,X_2,\hdots\}$ and $\mathbf{Y}=\{Y_1,Y_2,\hdots\}$, their respective SPAS- satisfy $\mathcal{S}_{\mathbf{X}}\succeq \mathcal{S}_{\mathbf{Y}}$ our general strategy is as follows.  For each $k\in\nats$ we find some $k'\in\nats$ such that for any graph $\Gamma$ and any $Y_k$-stable graph-like $\Gamma$-partition $\gamma\in\mathcal{P}(V^k)$, there is some $X_{k'}$-stable graph-like $\Gamma$-partition $\gamma'\in\mathcal{P}(V^{k'})$ such that $\pr_2\gamma' \succeq \pr_2\gamma$. By taking $\gamma'$ to be $[\alpha_{k,\Gamma}]^{X_{k'}}$ (which is graph-like by Proposition \ref{ug}), we then have that $\overline{X}_k(\Gamma)=\pr_2\gamma' \succeq \pr_2\gamma\succeq \overline{Y}_{k}(\Gamma)$ for all graphs $\Gamma$. The last inequality follows from the fact that $\overline{Y}_{k}(\Gamma)$ is the $2$-projection of a minimal $Y_{k}$-stable partition refining $\alpha_{k,\Gamma}$.

For the proofs in this section, the operators $\WL_{k,r}$ and $\C_{k,r}$ are only considered in the case $k> r$, since the statements hold trivially when $k\leq r$.
\begin{lemma}\label{pf1}
For all $k,r\in\mathbb{N}$, any $\WL_{k,r}$-stable partition is $\C_{k,r}$-stable.
\end{lemma}
\begin{proof}
Let $\gamma\in\mathcal{P}(V^k)$ be $\WL_{k,r}$-stable. The size of $\{\vec{x}\in V^r \mid \gamma(\vec{v}\langle \vec{i},\vec{x}\rangle)=\phi_{\vec{i}}\,\,\forall\vec{i}\in [k]^{(r)}\}$ is then independent of the choice of $\vec{v}$ from the equivalence class $[\vec{v}]_{\gamma}$ for all $\vec{\phi}\in\Phi^{\gamma,r}$, by Proposition \ref{stab}. For each $\vec{j}\in [k]^{(r)}$ and $\sigma\in\im(\gamma)$, the following holds:
$$
\{\vec{x}\in V^r \mid \gamma(\vec{v}\langle \vec{j},\vec{x}\rangle)=\sigma\}=\bigcup_{\{\vec{\phi}\in\Phi^{\gamma,r} \mid \phi_{\vec{j}}=\sigma\}}\{\vec{x}\in V^r \mid \gamma(\vec{v}\langle \vec{i},\vec{x}\rangle)=\phi_{\vec{i}}\,\,\forall\vec{i}\in [k]^{(r)}\}.
$$
The right-hand side of the above is a disjoint union of sets whose sizes are independent of the choice of $\vec{v}$ from the equivalence class $[\vec{v}]_{\gamma}$. Hence, the size of the left-hand side, is also the same for any choice of $\vec{v}$ from the class $[\vec{v}]_\gamma$, implying that $\gamma$ is $\C_{k,r}$-stable.
\end{proof}

\begin{corollary}\label{pf4}
For any graph $\Gamma$ and $k,r\in\nats$
$$
\overline{\C}_{k,r}(\Gamma)\preceq\overline{\WL}_{k,r}(\Gamma).
$$
\end{corollary}

\begin{lemma}\label{lm1}
For all $k,r\in \nats$, any $\C_k$-stable partition is $\C_{k,r}$-stable.
\end{lemma}
\begin{proof}
Let $\gamma\in\mathcal{P}(V^k)$ be $\C_k$-stable. Then, by definition, it is $\C_{k,1}$ stable.

Suppose $\gamma$ is $\C_{k,m}$-stable for some $m<k$. Let $\vec{j}\in [k]^{(m)}$. Since $m<k$, there is some $j'\in [k]$ such that $j'\neq j_i$ for all $i\in [m]$. For all $\sigma,\tau\in \im(\gamma)$ and $\vec{v}\in V^k$ such that $\gamma(\vec{v})=\sigma$ define
$$
p_{\sigma\tau}=|\{u\in V \mid \gamma(\vec{v}\langle j',u\rangle)=\tau\}|
$$
and
$$
q_{\sigma\tau}=|\{\vec{w}\in V^m \mid \gamma(\vec{v}\langle \vec{j},\vec{w}\rangle)=\tau\}|.
$$

Because $\gamma$ is $\C_k$-stable and, by induction hypothesis, also $\C_{k,m}$-stable, $p_{\sigma\tau}$ and $q_{\sigma\tau}$ are independent of the choice of $\vec{v}$ from the equivalence class $[\vec{v}]_\gamma$.

Observe that
$$
\{\vec{w}\cdot u \in V^{m+1} \mid \gamma(\vec{v}\langle \vec{j}\cdot j',\vec{w}\cdot u\rangle)=\tau\}=\{\vec{w}\cdot u\in V^{m+1} \mid \gamma((\vec{v}\langle \vec{j},\vec{w}\rangle)\langle j',u\rangle)=\tau\}.
$$
From this one deduces that
$$
|\{\vec{w}\cdot u\in V^{m+1} \mid \gamma(\vec{v}\langle \vec{j}\cdot j',\vec{w}\cdot u\rangle)=\tau\}|=\sum_{\alpha\in \im(\gamma)} p_{\sigma \alpha}q_{\alpha\tau}
$$
The right-hand side of the above is independent of the choice of $\vec{v}$ from the equivalence class $[\vec{v}]_\gamma$. Hence, $\gamma$ is $\C_{k,m+1}$-stable.  Thus, by induction, it is $\C_{k,r}$ stable for all $r$.
\end{proof}

\begin{corollary}\label{pf2}
For any graph $\Gamma$ and $k,r\in\nats$
$$
\overline{\C}_k(\Gamma)\succeq \overline{\C}_{k,r}(\Gamma).
$$
\end{corollary}

Note that Lemmata \ref{pf1} and \ref{lm1} are not restricted to graph-like partitions.

\begin{lemma}\label{pff1}
For all $k,r\in\nats$, the $k$-projection of a $\C_{k+r,r}$-stable graph-like partition is $\WL_{k,r}$-stable.
\end{lemma}
\begin{proof}
Let $\vec{v},\vec{v}'\in V^k$ and let $\overline{\gamma}=\pr_k\gamma$ for some $\C_{k+r,r}$-stable graph-like partition $\gamma\in\mathcal{P}(V^{k+r})$. Because $\gamma$ is graph-like, by Lemma \ref{veryuseful} it follows that for all $\vec{i}\in [k+r]^{(r)}$
$$
\gamma(\vec{v}\langle \vec{i},\pr_{\vec{j}}\vec{v}\rangle)=\gamma(\vec{v}'\langle \vec{i},\pr_{\vec{j}}\vec{v}'\rangle)
$$
where $\vec{j}=(k+1,k+2,\hdots,k+r)\in [k+r]^{(r)}$. In particular, for all $\vec{i}\in [k]^{(r)}$
\begin{equation}\label{deff1}
\overline{\gamma}(\pr_k\vec{v}\langle \vec{i},\pr_{\vec{j}}\vec{v}\rangle)=\overline{\gamma}(\pr_k\vec{v}'\langle \vec{i},\pr_{\vec{j}}\vec{v}'\rangle).
\end{equation}

For all $\vec{v}\in V^{k+r}$ define $\vec{\Delta}\in (\Phi^{\overline\gamma,r})^{\im(\gamma)}$ to be
$$
(\Delta_{\gamma(\vec{v})})_{\vec{i}}=\overline{\gamma}(\pr_k\vec{v}\langle \vec{i},\pr_{\vec{j}}\vec{v}\rangle).
$$
One deduces from formula (\ref{deff1}) that $\vec{\Delta}$ is well defined.

For any $\vec{\phi}\in \Phi^{\overline\gamma,r}$, it follows from the definition of $\vec{\Delta}$ that
$$
\{\vec{u}\in V^r \mid \overline{\gamma}(\pr_k\vec{v}\langle \vec{i},\vec{u}\rangle)=\phi_{\vec{i}} \,\,\forall \vec{i}\in [k]^{(r)}\}=\bigcup_{\{\sigma\in\im(\gamma) \mid \Delta_\sigma=\vec{\phi}\}}\{\vec{u}\in V^r \mid \gamma(\vec{v}\langle\vec{j},\vec{u}\rangle)=\sigma\}.
$$
$\C_{k,r}$-stability of $\gamma$ implies that the size of the right-hand side of the above is independent of the choice of $\vec{v}$ from the equivalence class $[\vec{v}]_\gamma$. Since $\gamma$ is graph-like, the size of the left-hand side is independent of the choice of $\vec{w}\in V^{k+r}$ such that $\pr_k\vec{w}\in[\pr_k\vec{v}]_{\overline\gamma}$, and the result follows.
\end{proof}

\begin{corollary}\label{cfiii}
For any graph $\Gamma$ and $k,r\in \nats$
$$
\overline{\C}_{k+r,r}(\Gamma)\succeq\overline{\WL}_{k,r}(\Gamma).
$$
In particular,
$$
\overline{\C}_{k+1}(\Gamma)\succeq\overline{\WL}_k(\Gamma).
$$
\end{corollary}

\begin{lemma}
For all $k,r\in\nats$, the $(k-r+1)$-projection of a $\C_{k,r}$-stable graph-like partition is $\C_{k-r+1}$ stable.
\end{lemma}
\begin{proof}
Let $\overline{\gamma}=\pr_{k-r+1}\gamma$ and let $\vec{v}\in V^k$. Fix some $i\in [k-r]$ and let $\vec{j}\in [k]^{(k-r+1)}$ be $\vec{j}=(i,k-r+1,k-r+2,\hdots,k-1,k)$. Fix some $u\in \vec{V}$ and let $\vec{w}=(v_{k-r+1},v_{k-r+1},\hdots,v_{k-r+1})\in V^{k-r}$.

Since $\gamma$ is graph-like, for any $u'\in V$
$$
\gamma(\vec{v}\langle \vec{j},u\cdot\vec{w}\rangle)=\gamma(\vec{v}\langle \vec{j},u'\cdot\vec{w}\rangle)\iff \overline{\gamma}(\pr_k\vec{v}\langle i,u\rangle)=\overline{\gamma}(\pr_k\vec{v}\langle i,u'\rangle)
$$
and therefore
$$
\{u'\in V \mid \overline{\gamma}(\pr_k\vec{v}\langle i,u\rangle)=\overline{\gamma}(\pr_k\vec{v}\langle i,u'\rangle)\}=\{u'\in V \mid \gamma(\vec{v}\langle \vec{j}, u'\cdot\vec{w}\rangle)=\gamma(\vec{v}\langle \vec{j},u\cdot\vec{w}\rangle)\}.
$$
The size of the right-hand side of the latter is independent of the choice of $\vec{v}$ from the equivalence class $[\vec{v}]_\gamma$. Hence, as $\gamma$ is graph-like, the size of the left-hand side is independent of the choice of $\vec{w}\in V^{k}$ such that $\pr_{k-r+1}\vec{w}\in[\pr_{k-r+1}\vec{v}]_{\overline\gamma}$. The result follows.
\end{proof}

\begin{corollary}\label{pf3}
For any graph $\Gamma$ and $k,r\in\nats$
$$
\overline{\C}_{k,r}(\Gamma)\succeq \overline{\C}_{k-r+1}(\Gamma).
$$
\end{corollary}

The results proved so far involved showing that the projection of a partition satisfying some stability condition satisfied some other stability condition. In order to prove that $\mathcal{S}_{\WL}$ dominates $\mathcal{S}_{\C}$ we \emph{extend} a $\WL_k$-stable partition to a $\C_{k+1}$-stable one. Let $\gamma$ be graph-like and be a $\WL_k$-stable partition of $V^k$. Define $\hat{\gamma}\in\mathcal{P}(V^{k+1})$ as follows:
$$
\hat{\gamma}(\vec{v},w)=(\gamma(\vec{v}),\gamma(\vec{v}\langle 1,w\rangle),\gamma(\vec{v}\langle 2,w\rangle),\hdots,\gamma(\vec{v}\langle k,w\rangle))\;\;\forall\vec{v}\in V^k,w\in V.
$$
That is to say, $\hat{\gamma}(\vec{v},w)=\hat{\gamma}(\vec{v}',w')$ if, and only if, $\gamma(\vec{v})=\gamma(\vec{v}')$ and $\gamma(\vec{v}\langle i,w\rangle)=\gamma(\vec{v}'\langle i,w'\rangle)$ for all $i\in [k]$. First, we need to prove that $\hat\gamma$ is graph-like.

\begin{lemma}
Let $\gamma$ and $\hat{\gamma}$ be as above. If $\gamma$ is invariant, then $\hat\gamma$ is invariant.
\end{lemma}
\begin{proof}
Because $\Sym(k+1)=\langle \Sym(k),(k,k+1)\rangle$ it is sufficient to show that  $\hat\gamma(\vec{u}\cdot u')=\hat\gamma(\vec{v}\cdot v')\implies \hat\gamma((\vec{u}\cdot u')^\tau)=\hat\gamma((\vec{v}\cdot v')^\tau)$, where $\tau=(k,k+1)$ is a transposition of $\Sym(k+1)$, or, equivalently
\begin{equation}\label{73}
\hat\gamma((\pr_{k-1}\vec{u})\cdot u'\cdot u_k)=\hat\gamma((\pr_{k-1}\vec{v})\cdot v'\cdot v_k).
\end{equation}

From the definition of $\hat\gamma$, it holds that
$$
\gamma(\vec{u}\langle k,u'\rangle)=\gamma(\vec{v}\langle k,v'\rangle)
$$
and hence
\begin{equation}\label{75}
\gamma((\pr_{k-1}\vec{u})\cdot u')=\gamma((\pr_{k-1}\vec{v})\cdot v').
\end{equation}
Also, $\hat\gamma(\vec{u}\cdot u')=\hat\gamma(\vec{v}\cdot v')$ implies that $\gamma(\vec{u})=\gamma(\vec{v})$ and hence, since $\gamma$ is graph-like
\begin{equation}\label{76}
\gamma((\pr_{k-1}\vec{u})\cdot u_k)=\gamma((\pr_{k-1}\vec{v})\cdot v_k).
\end{equation}
Since $\gamma(\vec{u}\langle i,u'\rangle)=\gamma(\vec{v}\langle i,v'\rangle)$ for all $i\in[k]$ and $\gamma$ is invariant, for any $\tau=(i,k)\in \Sym(k)$, $\gamma(\vec{u}\langle i,u'\rangle ^\tau)=\gamma(\vec{v}\langle i,v'\rangle^\tau)$ or, equivalently
\begin{equation}\label{77}
\gamma((\pr_{k-1}\vec{u})\cdot u'\langle i,u_k\rangle)=\gamma((\pr_{k-1}\vec{v})\cdot v'\langle i, v_k\rangle).
\end{equation}
Combining formulae (\ref{75}), (\ref{76}), and (\ref{77}), formula (\ref{73}) follows.
\end{proof}
Consequently, if $\gamma$ is graph-like, so is $\hat\gamma$. Indeed, $\hat\gamma$ is $k$-consistent by construction. Also, since $\hat\gamma$ is invariant and $\gamma$ satisfies formula (\ref{prop}), then $\hat\gamma$ also satisfies formula (\ref{prop}). Finally, observe that if $\gamma$ is a $\Gamma$-partition for some graph $\Gamma$, then so is $\hat\gamma$.

\begin{lemma}
Let $\gamma$ and $\hat{\gamma}$ be as above. If $\gamma$ is graph-like and $\WL_k$ stable, then $\hat{\gamma}$ is $\C_{k+1}$-stable.
\end{lemma}
\begin{proof}
As observed, if $\gamma$ is graph-like, then $\hat\gamma$ is graph-like by construction. Since $\gamma$ is $\WL_k$-stable, it follows that for all $\sigma\in\im(\hat\gamma)$ and $\vec{v}\in V^{k+1}$, the size of $\{x\in V \mid \hat\gamma(\vec{v}\langle k+1,x\rangle)=\sigma\}$ is independent of the choice of $\vec{v}$ from the equivalence class $[\vec{v}]_{\hat\gamma}$. By invariance of $\hat\gamma$, one deduces that for all $i\in[k+1]$ the size of $\{x\in V \mid \hat\gamma(\vec{v}\langle i,x\rangle)=\sigma\}$ is independent of the choice of $\vec{v}$ in the equivalence class $[\vec{v}]_{\hat\gamma}$. Thus, $\hat\gamma$ is $\C_{k+1}$-stable.
\end{proof}

\begin{corollary}
For any graph $\Gamma$
$$
\overline{\C}_{k+1}(\Gamma)\succeq \overline{\WL}_k(\Gamma).
$$
\end{corollary}

In particular, combining this with Corollary \ref{cfiii} one deduces that for all graphs $\Gamma$ and $k\in \nats$
\begin{equation}\label{cf}
\overline{\C}_{k+1}(\Gamma)\approx\overline{\WL}_k(\Gamma).
\end{equation}
This can be seen as a combinatorial reformulation of Theorem 5.2 in \cite{cfi} proved without referring to the Immerman-Lander pebble game.

\begin{proof}[Proof of Theorem \ref{thm:main1}]
Let $\Gamma$ be any graph. From formula (\ref{cf}) it easily follows that $\mathcal{S}_{\mathbf{WL}}\simeq\mathcal{S}_{\mathbf{C}}$.

It follows from Corollaries \ref{pf2} and \ref{pf3} that for all $r,k\in \nats$
$$
\overline{\C}_k(\Gamma)\succeq \overline{\C}_{k,r}(\Gamma)\succeq \overline{\C}_{k-r+1}(\Gamma)
$$
and hence for any $r\in \nats$
$$
\mathcal{S}_{\mathbf{C}}\simeq\mathcal{S}_{\mathbf{C},r}.\\
$$

From Corollaries \ref{cfiii} and \ref{pf4}, we deduce that for all $r\in\nats$.
$$
\mathcal{S}_{\mathbf{WL},r}\simeq\mathcal{S}_{\mathbf{C},r}
$$
\end{proof}

\section{The Evdokimov-Ponomarenko SPAS}

In this section we apply the language and results of this paper to derive the following statement about the SPAS $\mathcal{S}_{\EP}=\{\EP_1,\EP_2,\hdots\}$, introduced by Evdokimov and Ponomarenko in \cite{pono}.
\begin{theorem}[Evdokimov, Ponomarenko, 1999]
For any graph $\Gamma$, and $k\in\nats$
\begin{equation}\label{86}
\overline{\WL}_k(\Gamma)\preceq \EP_k(\Gamma)\preceq \overline{\WL}_{3k}(\Gamma).
\end{equation}
In particular, $\mathcal{S}_{\EP}\simeq\mathcal{S}_{\mathbf{WL}}$.
\end{theorem}
We first define the mapping $\EP_k$ in terms of the refinement operator $\WL_2$.

For $P\in \mathcal{P}(V^2)$, $k\in\nats$ and some symbol $\Delta$, set $P^{(k)}: (V^k)^2\rightarrow \im(P)^k\cup\im(P)^k\times\{\Delta\}$ to be the following labelled partition of $(V^k)^2$:
$$
P^{(k)}(\vec{u},\vec{v})=\begin{cases}
(P(u_1,v_1),P(u_2,v_2),\hdots,P(u_k,v_k),\Delta)\,\;&\text{if $\vec{u}=\vec{v}=(u,u,\hdots,u)$ for some $u\in V$}\\
(P(u_1,v_1),P(u_2,v_2),\hdots,P(u_k,v_k))\,\;&\text{otherwise}
\end{cases}
$$
for all $\vec{u},\vec{v}\in V^k$. Note that we have denoted elements of $(V^k)^2$ by $(\vec{u},\vec{v})$ as opposed to $(\vec{u}\cdot\vec{v})$ to emphasize the difference between a set of \emph{pairs} of $k$-tuples as opposed to a set of $2k$-tuples. Hence, if $\Gamma$ is a graph, then $[\Gamma^{(k)}]^{\WL_2}$ is a coherent configuration on $V^k$, which we denote by $\hat{\Gamma}^{(k)}$. Observe that the set
$$
I_\Delta=\{ (u,u,\hdots,u) \mid u\in V\}\subset V^k
$$
is a union of cells of $\hat{\Gamma}^{(k)}$. Hence, the restriction $\hat{\Gamma}^{(k)}|_{I_\Delta^2}$ is a coherent configuration on $I_\Delta$. Let $\delta:V^2\rightarrow I_\Delta^2$ be the map
\begin{equation}\label{89}
(u,v)\mapsto (\vec{u},\vec{v})
\end{equation}
where $\vec{u}=(u,u,\hdots,u)\in V^k$ and $\vec{v}=(v,v,\hdots,v)\in V^k$. We define $\EP_k(\Gamma)=\hat{\Gamma}^{(k)}\circ\delta$ which, one may note, is a coherent configuration on $V$.

Theorem~1.1 in~\cite{pono4} shows that the family $\mathcal{S}_{\EP}=\{\EP_1,\EP_2,\hdots\}$ forms, indeed, a SPAS. Furthermore, the authors also prove the following properties. For a binary relation $R$ on $V$, set $X_R^k=\{(u,v,\hdots,v)\in V^k \mid (u,v)\in R\}$.
\begin{proposition}[Proposition 3.6 in \cite{pono4}]\label{pono3}
$R=\{(u,v)\in V^2 \mid \EP_k(\Gamma)(u,v)=\sigma\}$ for some $\sigma\in\im(\EP_k(\Gamma))$ if, and only if, $X_R^k$ is a cell of $\hat\Gamma^{(k)}$.
\end{proposition}
\begin{proposition}[Proposition 3.6 in \cite{pono4}]\label{pono2}
The equivalence classes of $[\alpha_{k,\Gamma}]^{\WL_k}$ are unions of cells of $\hat\Gamma^{(k)}$.
\end{proposition}
From these, one can deduce the left-most relation of formula (\ref{86}).
\begin{lemma}
For any graph $\Gamma$ and $k\in\nats$
$$
\overline{\WL}_k(\Gamma)\preceq \EP_k(\Gamma).
$$
\end{lemma}
\begin{proof}
For $k=1$ the statement is trivial. Let $k\geq 2$ and set $\Lambda=[\alpha_{k,\Gamma}]^{\WL_k}$. By Proposition \ref{pono2}, and because $\Lambda$ is graph-like, for any $u,v\in V$ the equivalence class $[(v,u,\hdots,u)]_{\Lambda}$ is a union of cells of $\hat\Gamma^{(k)}$ whose elements are all of the form $(v',u',\hdots,u')$,. Since $\overline{\WL}_k(\Gamma)=\pr_2\Lambda$, it follows from Proposition \ref{pono3} that its equivalence classes are unions of equivalence classes of $\EP_k(\Gamma)$. Hence $\overline{\WL}_k(\Gamma)\preceq \EP_k(\Gamma)$.
\end{proof}
We now apply Lemmata \ref{pf1} and \ref{pff1} to show the rightmost relation of formula (\ref{86}). For all $k,p\in\nats$ let $\psi_{k,p}:V^{pk}\rightarrow (V^k)^p$ be the map
$$
\vec{v}\mapsto (\vec{w}_1,\vec{w}_2,\hdots,\vec{w}_p)
$$
where $\vec{w}_i=(v_{1+(i-1)k},v_{2+(i-1)k},\hdots,v_{k+(i-1)k})$.
\begin{lemma}
For any graph $\Gamma$ and $k\in\nats$
$$
\EP_k(\Gamma)\preceq \overline{\WL}_{3k}(\Gamma).
$$
\end{lemma}
\begin{proof}
Let $\Lambda=[\alpha_{3k,\Gamma}]^{\C_{3k,k}}$ and set $\hat\Phi=\Phi^{\pr_{2k}\Lambda,k}$ . By Lemma \ref{pff1} $\pr_{2k}\Lambda$ is a $\WL_{2k,k}$ stable partition of $V^{2k}$. Hence, for all $\vec{\phi}\in \hat\Phi$ the size of
$$
\{\vec{x}\in V^k \mid \pr_{2k}\Lambda(\vec{v}\langle \vec{i},\vec{x}\rangle)=\phi_{\vec{i}}\,\,\forall \vec{i}\in [2k]^{(k)}\}
$$
is independent of the choice of $\vec{v}$ from the equivalence class $[\vec{v}]_{\pr_{2k}\Lambda}$. In particular, if $T=\{(1,2,\hdots,k),(k+1,k+2,\hdots,2k)\}$, then for all $\vec{\xi}\in\im(\pr_{2k}\Lambda)^{T}$
$$
\{\vec{x}\in V^k \mid \pr_{2k}\Lambda(\vec{v}\langle \vec{i},\vec{x}\rangle)=\xi_i\,\,\forall \vec{i}\in T\}=\bigcup_{\{\vec{\phi}\in \hat\Phi \mid \phi_{\vec{i}}=\xi_{\vec{i}}\,\,\forall\vec{i}\in T\}}\{\vec{x}\in V^k \mid \pr_{2k}\Lambda(\vec{v}\langle \vec{i},\vec{x}\rangle)=\phi_{\vec{i}}\,\,\forall \vec{i}\in [2k]^{(k)}\}.
$$
The right-hand side of the above is a disjoint union of sets whose size is independent of the choice of $\vec{v}$ from the equivalence class $[\vec{v}]_{\pr_{2k}\Lambda}$. Hence, the size of the left-hand side is independent of the choice of $\vec{v}$ from the equivalence class $[\vec{v}]_{\pr_{2k}\Lambda}$, and therefore, $(\pr_{2k}\Lambda)\circ \psi_{k,2}^{-1}$ is a $\WL_2$-stable partition of $(V^k)^2$. In particular, it refines $\alpha_{2k,\Gamma}\circ \psi_{k,2}^{-1}$. But $\Gamma^{(k)}\preceq \alpha_{2k,\Gamma}\circ \psi_{k,2}^{-1}$ and thus, since $\hat\gamma^{(k)}$ is a minimal $\WL_2$-stable partition of $V^{2k}$ refining $\Gamma^{(k)}$, then 
$\hat\Gamma^{(k)}\preceq (\pr_{2k}\Lambda)\circ\psi_{k,2}^{-1}$. In particular, $\hat\Gamma^{(k)}|_{I_\Delta^2}\preceq (\pr_{2k}\Lambda)\circ\psi_{k,2}^{-1}|_{I_\Delta^2}$ and hence, if $\delta$ is as defined in formula (\ref{89})
$$
\EP_k(\Gamma)\preceq (\pr_{2k}\Lambda)\circ\psi_{k,2}^{-1}\circ \delta.
$$
Since $\Lambda$ is graph-like, Lemma \ref{veryuseful} implies that for all $u,u',v,v'\in V$
\begin{equation}\label{96}
\pr_{2k}\Lambda(u,v,\hdots,v)=\pr_{2k}\Lambda(u',v',\hdots,v')\implies \pr_{2k}\Lambda(\vec{u}\cdot\vec{v})=\pr_{2k}\Lambda(\vec{u}'\cdot\vec{v}),
\end{equation}
where $\vec{u}=(u,u,\hdots,u),\vec{v}=(v,v,\hdots,v),\vec{u}'=(u',u',\hdots,u'),\vec{v}'=(v',v',\hdots,v')\in V^k$. It therefore follows that $\pr_2\Lambda=(\pr_{2k}\Lambda) \circ \psi_{k,2}^{-1}\circ\delta$. But $\pr_2\Lambda=\overline{\C}_{3k,k}(\Gamma)\preceq \overline{\WL}_{3k}(\Gamma)$. Thus, by formula (\ref{96})
$$
\EP_k(\Gamma)\preceq \overline{\WL}_{3k}(\Gamma).
$$

\end{proof}

\section{Proofs of Theorems $\ref{thm:main2}$ and $\ref{thm:main3}$}

We proceed using the same strategy as for the proof of Theorem \ref{thm:main1}.

\begin{lemma}\label{im1}
The $k$-projection of a graph-like $\IM_{k+1}^\mathbb{F}$-stable partition is $\C_k$-stable for any field $\mathbb{F}$ and $k\in\nats$.
\end{lemma}
\begin{proof}
Let $\gamma\in\mathcal{P}(V^{k+2})$ be a graph-like $\IM_{k+2}^\mathbb{F}$-stable partition and set $\overline{\gamma}=\pr_k\gamma$. Suppose $\gamma(\vec{v})=\gamma(\vec{v}')$. Fix some $j\in[k]$ and let $\vec{i}=(j,k+1)\in[k+1]^{(2)}$. As $\gamma$ is graph-like, there is some $\sigma\in \im(\gamma)$ such that the matrix $\chi_{\sigma,\vec{i}}^{\gamma,\vec{v}}$ is diagonal and non-zero. Since $\gamma(\vec{v})=\gamma(\vec{v}')$, there is some $S\in \mathrm{GL}_V(\mathbb{F})$ such that $S\chi_{\sigma,\vec{i}}^{\gamma,\vec{v}}S^{-1}=\chi_{\sigma,\vec{i}}^{\gamma,\vec{v}'}$. So $\chi_{\sigma,\vec{i}}^{\gamma,\vec{v}}$ and $\chi_{\sigma,\vec{i}}^{\gamma,\vec{v}'}$ have the same rank, and hence the same number of $1$s on the diagonal. Thus, for any $w\in V$ the size of the set $\{u\in V \mid \gamma(\vec{v}\langle \vec{i},(u,u)\rangle)=\sigma\}$ is independent of the choice of $\vec{v}$ from the equivalence class $[\vec{v}]_\gamma$. Also,
$$
\{u\in V \mid \overline{\gamma}(\pr_k\vec{v}\langle j,u\rangle)=\sigma\}=\{u\in V \mid \gamma(\vec{v}\langle \vec{i},(u,u)\rangle)=\sigma\}
$$
and hence, as $\gamma$ is graph-like, the size of the left-hand side is independent of the choice of $\pr_k\vec{v}$ from the equivalence class $[\pr_k\vec{v}]_{\overline\gamma}$. The result follows.
\end{proof}
\begin{corollary}\label{imm1}
For any graph $\Gamma$ and $k\in\nats$
$$
\overline{\IM}_{k+2}^\mathbb{F}(\Gamma)\succeq \overline{\WL}_k(\Gamma).
$$
\end{corollary}
\begin{proof}
This follows from Lemma \ref{im1} and the fact that $\overline{\C}_{k+1}(\Gamma)\approx\overline{\WL}_k(\Gamma)$.
\end{proof}
\begin{lemma}\label{im2}
The $k$-projection of a graph-like $\C_{k+1}$-stable partition is $\IM_k^\mathbb{F}$-stable for any $\mathbb{F}$ with $\text{char}(\mathbb{F})=0$ or $\text{char}(\mathbb{F})> |V|$.
\end{lemma}
\begin{proof}
Let $\gamma\in\mathcal{P}(V^{k+1})$ be a graph-like $\C_{k+1}$-stable partition and fix some $\vec{i}\in[k+1]^{(3)}$ with $i_3=k+1$. For every $\vec{v}\in V^{k+1}$ define $g_{\vec{v}}$ to be the partition of $V^3$ given by $g_{\vec{v}}(\vec{x})=\gamma(\vec{v}\langle \vec{i},\vec{x}\rangle)$ for all $\vec{x}\in V^3$. Since $\gamma$ is $\C_{k+1}$-stable, it follows that $g_{\vec{v}}$ is $\C_3$ stable for all $\vec{v}\in V^{k+1}$. As $\gamma$ is graph-like, $\pr_2g_{\vec{v}}$ is a rainbow, and it is $\WL_2$-stable by Lemma \ref{pf1} and thus, a coherent configuration on $V$.

Set $\overline{\gamma}=\pr_k\gamma$ and for all $\vec{y}\in V^2$ let $\overline{g}_{\vec{v}}(\vec{y})=\overline{\gamma}(\pr_k\vec{v}\langle(i_1,i_2),\vec{y}\rangle)$. Then $\overline{g}_{\vec{v}}=\pr_2g_{\vec{v}}$. Therefore, all non-empty relations of $\overline{g}_{\vec{v}}$ form a coherent configuration whose $\mathbb{F}$-adjacency algebra has standard basis
$$
\{\chi^{\overline\gamma,\pr_k\vec{v}}_{(i_1,i_2),\sigma} \mid \exists \vec{y}\in V^2,\overline{\gamma}(\pr_k\vec{v}\langle (i_1,i_2),\vec{y}\rangle)=\sigma\}.
$$
Thus, if $\overline{\gamma}(\pr_k\vec{u})=\overline{\gamma}(\pr_k\vec{w})$ then $\overline{g}_{\vec{u}}$ and $\overline{g}_{\vec{w}}$ are algebraically isomorphic coherent configurations. More precisely, one can check that $\im(\overline{g}_{\vec{u}})=\im(\overline{g}_{\vec{v}})$ and that the map
$$
\begin{matrix}
\iota:\im(\overline{g}_{\vec{u}})& \rightarrow & \im(\overline{g}_{\vec{v}})\\
\sigma & \mapsto & \sigma
\end{matrix}
$$
is an algebraic isomorphism.

For $\mathrm{char}(\mathbb{F})=0$ or $\mathrm{char}(\mathbb{F})>|V|$ it follows from Corollary \ref{newbound} that there is some $S\in \mathrm{GL}_V(\mathbb{F})$ such that $S\chi_{(i_1,i_2),\sigma}^{\overline\gamma,\pr_k\vec{u}}S^{-1}=\chi_{(i_1,i_2),\sigma}^{\overline\gamma,\pr_k\vec{w}}$ for all $\sigma\in \im(\overline{\gamma})$. The result then follows.
\end{proof}
Note that in the above proof, there may be other bijections $\im(\overline{g}_{\vec{u}})\rightarrow\im(\overline{g}_{\vec{v}})$ which are algebraic isomorphisms. However, it follows from the definition of the $k$-refinement operator $\IM_k^\mathbb{F}$, that $\vec{u},\vec{v}\in V^k$ are in the same equivalence class of an $\IM_k^\mathbb{F}$-stable partition, only if $\iota$ as above is an algebraic isomorphism.
\begin{corollary}\label{imm2}
For all $k\in \nats$, graphs $\Gamma$ with vertex set $V$ and fields $\mathbb{F}$ such that $\mathrm{char}(\mathbb{F})=0$ or $\mathrm{char}(\mathbb{F})>|V|$.
$$
\overline{\WL}_k(\Gamma)\succeq \overline{\IM}_k^\mathbb{F}(\Gamma).
$$
\end{corollary}
\begin{proof}
This follows from Lemma \ref{im2} and the fact that $\overline{\C}_{k+1}(\Gamma)\approx\overline{\WL}_k(\Gamma)$.
\end{proof}
The statement of Theorem \ref{thm:main2} comes from taking $\mathbb{F}$ of characteristic $0$ in Corollaries \ref{imm1} and \ref{imm2}. Theorem \ref{thm:main3} arises instead from combining Corollary \ref{imm2} with a construction due to Holm~\cite{bjarki}.  The construction in the proof of Theorem 7.1 in~\cite{bjarki} gives, for each $k\in\mathbb{N}$ and prime number $p$, a graph $\Gamma_{k,p}$ for which $\overline{\WL}_k(\Gamma_{k,p})$ is \emph{strictly coarser} than $\Sch(\Gamma_{k,p})$ but $\overline{\IM}^{\mathbb{F}}_3(\Gamma_{k,p})=\Sch(\Gamma_{k,p})$, for any field $\mathbb{F}$ with $\mathrm{char}(\mathbb{F}) = p$.  Furthermore, from the same result, it also follows that $\overline{\IM}_k^{\mathbb{F}}(\Gamma_{k,p})$ is strictly coarser than $\Sch(\Gamma_{k,p})$ whenever $\mathrm{char}(\mathbb{F})\neq p$. This shows that the SPAS $\mathcal{S}_{\IM}({\mathbb{F}_1})$ and $\mathcal{S}_{\IM}({\mathbb{F}_2})$ are incomparable whenever $\mathrm{char}(\mathbb{F}_1)\neq \mathrm{char}(\mathbb{F}_2)$.

\section{Yet another refinement operator}
There is a subtle difference between the definitions of $\WL_{k,r}$ and $\C_{k,r}$: the colours of $\WL_{k,r}\circ\gamma$ are \emph{multisets of tuples} of colours of $\gamma$, whereas the colours of $\C_{k,r}\circ\gamma$ are \emph{tuples of multisets} of colours of $\gamma$. We now show that a similar variation in the definition of $\IM_k^{\mathbb{F}}$ gives a refinement procedure whose corresponding SPAS is equivalent to $\mathcal{S}_{\mathbf{IM}(\mathbb{F})}$.

For every $\gamma\in\mathcal{P}(V^k),\vec{v}\in V^k$ and $\vec{\phi}\in \Phi^{\gamma,2}$ let $\chi^{\gamma,\vec{v}}_{\vec{\phi}}$ be the adjacency matrix of the relation $\{\vec{x}\in V^2 \mid \gamma(\vec{v}\langle \vec{i},(u,v)\rangle)=\phi_{\vec{i}}\,\,\forall \vec{i}\in[k]^{(2)}\}\subseteq V^2$. Define the mapping $\IMt_k^\mathbb{F}$ by setting $\IMt^{\mathbb{F}}_1\circ\gamma=\IMt^\mathbb{F}_2\circ\gamma=\gamma$ and for $k> 2$:
$$
\begin{matrix}
\IMt_k^{\mathbb{F}}\circ\gamma: & V^k & \rightarrow & \im(\gamma)\times(\mathrm{Mat}_V(\mathbb{F})^{\Phi^{\gamma,2}}/\sim)\\
& \vec{v} &\mapsto & (\gamma(\vec{v}),(\chi^{\gamma,\vec{v}}_{\vec{\phi}})_{\vec{\phi}\in \Phi^{\gamma,2}})
\end{matrix}
$$
where the equivalence classes of the relation $\sim$ are the orbits of $GL_{V}(\mathbb{F})$ acting on the tuples by conjugation.

Similarly to Proposition \ref{ug}, one can show that $\IMt_k^\mathbb{F}$ is a graph-like $k$-refinement operator for all $k\in\nats$ and that $\mathbf{IMt}(\mathbb{F})=\{\IMt^{\mathbb{F}}_1,\IMt^{\mathbb{F}}_2,\hdots\}$ is a refinement procedure. Hence, the family $\mathcal{S}_{\mathbf{IMt}(\mathbb{F})}=\{\overline{\IMt}^{\mathbb{F}}_1,\overline{\IMt}^{\mathbb{F}}_2,\hdots\}$ is a SPAS for any field $\mathbb{F}$. Also, similarly to Proposition \ref{stab}, one may derive the following stability condition.
\begin{proposition}
A partition $\gamma\in\mathcal{P}(V^k)$ is $\IMt^{\mathbb{F}}_k$-stable if, and only if, for all $\vec{u},\vec{v}\in V^k$
$$
\gamma(\vec{u})=\gamma(\vec{v})\implies \exists S\in \mathrm{GL}_V(\mathbb{F}),\;\forall\vec{\phi}\in \Phi^{\gamma,2}\,S\chi^{\gamma,\vec{u}}_{\vec{\phi}}S^{-1}=\chi^{\gamma,\vec{v}}_{\vec{\phi}}. 
$$
\end{proposition}
In particular, the following result is analogous to Lemma \ref{stab}. 
\begin{lemma}\label{stab2}
Any $\IMt^{\mathbb{F}}_k$-stable partition is also $\IM_k^{\mathbb{F}}$-stable.
\end{lemma}
\begin{proof}
For any $\vec{v}\in V^k,\vec{i}\in [k]^{(2)}$ and $\sigma \in \im(\gamma)$
$$
\chi^{\gamma,\vec{v}}_{\sigma,\vec{i}}=\sum_{\{\vec{\phi}\in \Phi^{\gamma,2} \mid \phi_{\vec{i}}=\sigma\}}\chi^{\gamma,\vec{v}}_{\vec{\phi}}.
$$
Hence if $\gamma(\vec{u})=\gamma(\vec{v})$, there is some $S\in \mathrm{GL}_V(\mathbb{F})$ such that $S\chi^{\gamma,\vec{v}}_{\vec{\phi}}S^{-1}=\chi^{\gamma,\vec{u}}_{\vec{\phi}}$ for all $\vec{\phi}\in \Phi^{\gamma,2}$. Thus
$$
S\chi^{\gamma,\vec{v}}_{\sigma,\vec{i}}S^{-1}=\sum_{\{\vec{\phi}\in \Phi^{\gamma,2} \mid \phi_{\vec{i}}=\sigma\}}S\chi^{\gamma,\vec{v}}_{\vec{\phi}}S^{-1}=\sum_{\{\vec{\phi}\in \Phi^{\gamma,2} \mid \phi_{\vec{i}}=\sigma\}}\chi^{\gamma,\vec{u}}_{\vec{\phi}}=\chi_{\sigma,\vec{i}}^{\gamma,\vec{u}}
$$
from which follows that $\gamma$ is $\IM^{\mathbb{F}}_k$-stable.
\end{proof}
\begin{corollary}\label{imt12}
For any graph $\Gamma$, field $\mathbb{F}$ and $k\in\nats$
$$
\overline{\IM}^{\mathbb{F}}_k(\Gamma)\preceq \overline{\IMt}^{\mathbb{F}}_k(\Gamma).
$$
\end{corollary}
In the following result, the argument, is analogous to that of Lemma \ref{pff1}.
\begin{lemma}
The $k$-projection of an $\IM_{k+2}^{\mathbb{F}}$-stable partition is $\IMt^{\mathbb{F}}_k$-stable.
\end{lemma}
\begin{proof}
Let $\vec{u},\vec{v}\in V^k$ be such that $\gamma(\vec{u})=\gamma(\vec{v})$, and set $\overline{\gamma}=\pr_k\gamma$. Because $\gamma$ is graph-like, it follows from Lemma \ref{veryuseful} that for all $\vec{i}\in[k+2]^{(2)}$
$$
\gamma(\vec{v}\langle \vec{i},\pr_{\vec{j}}\vec{v}\rangle)=\gamma(\vec{u}\langle \vec{i},\pr_{\vec{j}}\vec{u}\rangle)
$$
where $\vec{j}=(k+1,k+2)\in[k+2]^{(2)}$. In particular, for all $\vec{i}\in[k]^{(2)}$
\begin{equation}\label{defim}
\overline{\gamma}(\pr_k\vec{v}\langle \vec{i},\pr_{\vec{j}}\vec{v}\rangle)=\overline{\gamma}(\pr_k\vec{u}\langle \vec{i},\pr_{\vec{j}}\vec{u}\rangle).
\end{equation}
Define $\vec{\Delta}\in (\Phi^{\overline\gamma,2})^{\im(\gamma)}$
$$
(\Delta_{\gamma(\vec{v})})_{\vec{i}}=\overline{\gamma}(\pr_k\vec{v}\langle \vec{i},\pr_{\vec{j}}\vec{v}\rangle).
$$
One deduces from formula (\ref{defim}) that $\vec{\Delta}$ is well defined.

For any $\vec{\phi}\in \Phi$
$$
\{\vec{u}\in V^2 \mid \overline{\gamma}(\pr_k\vec{v}\langle \vec{i},\vec{u}\rangle)=\phi_{\vec{i}} \,\,\forall \vec{i}\in [k]^{(2)}\}=\bigcup_{\{\sigma\in\im(\gamma) \mid \Delta_\sigma=\vec{\phi}\}}\{\vec{u}\in V^2 \mid \gamma(\vec{v}\langle\vec{j},\vec{u}\rangle)=\sigma\}.
$$
Hence, for all $\vec{\phi}\in \Phi$ it holds that
$$
\chi^{\gamma,\vec{v}}_{\vec{\phi}}=\sum_{\{\sigma\in\im(\gamma) \mid \Delta_\sigma=\vec{\phi}\}} \chi^{\gamma,\vec{v}}_{\sigma,\vec{j}}.
$$
Because $\gamma$ is $\IM_{k+2}^{\mathbb{F}}$-stable, and $\gamma(\vec{u})=\gamma(\vec{v})$, for each $\vec{i}\in [k]^{(2)}$ there is some $S\in \mathrm{GL}_V(\mathbb{F})$ such that $S\chi_{\sigma,\vec{i}}^{\gamma,\vec{v}}S^{-1}=\chi_{\sigma,\vec{i}}^{\gamma,\vec{u}}$ for all $\sigma\in \im(\gamma)$. Thus, for all $\vec{\phi}\in \im(\overline{\gamma})^{[k]^{(2)}}$:
$$
S\chi_{\vec{\phi}}^{\gamma,\vec{v}}S^{-1}=\sum_{\{\sigma\in\im(\gamma) \mid \Delta_\sigma=\vec{\phi}\}}S\chi_{\sigma,\vec{j}}^{\gamma,\vec{v}}S^{-1}=\sum_{\{\sigma\in\im(\gamma) \mid \Delta_\sigma=\vec{\phi}\}}\chi^{\vec{u}}_{\sigma,\vec{j}}=\chi_{\vec{\phi}}^{\gamma,\vec{u}},
$$
whence, $\overline{\gamma}(\vec{u})=\overline{\gamma}(\vec{v})$ and therefore $\overline{\gamma}$ is $\IMt^{\mathbb{F}}_k$-stable.
\end{proof}
\begin{corollary}\label{imt11}
For any graph $\Gamma$, field $\mathbb{F}$ and $k\in\nats$
$$
\overline{\IMt}^{\mathbb{F}}_k(\Gamma)\preceq \overline{\IM}_{k+2}^{\mathbb{F}}(\Gamma).
$$
\end{corollary}
Theorem \ref{thm:main4} follows from Corollaries \ref{imt11} and \ref{imt12}.

\section{Conclusions}
The Weisfeiler-Leman algorithm is much studied in the context of graph
isomorphism.  It is really a family of algorithms, graded by a
dimension parameter.  A large number of other families of algorithms
have been shown to give essentially the same graded approximations of
isomorphism.  The Schurian polynomial approximation schemes of
Evdokimov et al.\ provide a general framework for comparing these
families of algorithms.
The invertible map operators of Dawar and Holm provide another such
family of algorithms (or, more formally in the language of this paper, \emph{refinement procedure}), but one that has greater distinguishing power than the Weisfeiler-Leman family.  In the same way as $\mathbf{WL}_r$ and $\mathbf{C}_r$ were obtained from $\mathbf{WL}$ and $\mathbf{C}$, one can generalize $\mathbf{IM}(\mathbb{F})$ as follows: for every $k,r\in\mathbb{N}$ define the $k$-refinement operator $\IM_{k,r}^\mathbb{F}$ by setting $\IM_{k,r}^{\mathbb{F}}\circ\gamma=\gamma$ when $k\leq 2r$.  When $k>2r$ define:
$$
\begin{matrix}
\IM_{k,r}^{\mathbb{F}}\circ\gamma: & V^k\times& \rightarrow & \im(\gamma)\times(\mathrm{Mat}_{V^r}(\mathbb{F})^{\im(\gamma)\times [k]^{(2r)}}/\sim)\\
& \vec{v}&\mapsto & (\gamma(\vec{v}),((\chi^{\gamma,\vec{v}}_{\vec{i},\sigma})_{\sigma\in \im(\gamma)})_{\vec{i}\in [k]^{(2r)}})
\end{matrix}
$$
where $\chi^{\gamma,\vec{v}}_{\vec{i},\sigma}$ is the adjacency matrix of the relation $\{(\vec{x},\vec{y}) \mid \gamma(\vec{v}\langle \vec{i},\vec{x}\cdot\vec{y}\rangle)=\sigma\}\subseteq (V^r)^2$ and $\sim$ is the relation whose equivalence classes are the orbits of $GL_{V^r}(\mathbb{F})$ acting on the tuples by conjugation. One can show that $\mathbf{IM}_r(\mathbb{F})=\{\IM_{1,r}^\mathbb{F},\IM_{2,r}^\mathbb{F},\hdots\}$ is a refinement procedure for all $r\in\nats$. One can thus derive from it a SPAS $\mathcal{S}_{\mathbf{IM}(\mathbb{F}),r}$ in the same manner as described in Section \ref{rop}.  While we were able to show that the refinement procedures $\mathbf{WL}_r$ and $\mathbf{C}_r$ do not yield SPAS' more powerful than that yielded by $\mathbf{WL}$,  the exact relation between $\mathcal{S}_{\mathbf{IM}(\mathbb{F})}$ and $\mathcal{S}_{\mathbf{IM}(\mathbb{F}),r}$ is still unclear and an interesting open question.

\bibliographystyle{alpha}
\bibliography{WL(1)}

\end{document}